\documentclass[10pt]{article}
\usepackage[utf8]{inputenc}

\usepackage{amsfonts,amsmath, amssymb,amsthm,amscd}
\usepackage{tikz}
\usepackage{color}
\usepackage{mathtools}
\usepackage{dsfont}
\usepackage{comment}
\usepackage{geometry}

\newtheorem{theorem}{Theorem}
\newtheorem{proposition}{Proposition}

\newtheorem{lemma}{Lemma}
\newtheorem{definition}{Definition}

\newtheorem{remark}{Remark}
\newtheorem{coro}[theorem]{Corollary}
\newtheorem{hypothesis}[theorem]{Hypothesis}

\newcommand{\E}{{\mathbb E }}

\renewcommand{\P}{{\mathbb P}}
 
\allowdisplaybreaks
\begin{document}

 \begin{minipage}{0.85\textwidth}
\end{minipage}
\begin{center}
	\large\bf
Rigidity of eigenvalues for $\beta$ ensemble in multi-cut regime
\end{center}

\begin{center} 
		\begin{minipage}{0.3\textwidth}
			\begin{center}
				Yiting Li  \\
				\footnotesize 
				{Korea Advanced Institute of Science and Technology}\\
				{\it yitingli@kaist.ac.kr}
			\end{center}
		\end{minipage} 
\end{center}

\begin{abstract}
For a $\beta$ ensemble on $\Sigma^{(N)}=\{(x_1,\ldots,x_N)\mathbb R^N|x_1\le\cdots\le x_N\}$ with real analytic potential and general $\beta>0$, under the assumption that its equilibrium measure is supported on $q$ intervals where $q>1$, we prove the following rigidity property for its particles.
\begin{enumerate}
	\item In the bulk of the spectrum,  with overwhelming probability, the distance between a particle and its classical position is of order $O(N^{-1+\epsilon})$.
	\item If $k$ is close to 1 or close to $N$, i.e., near the extreme edges of the spectrum, then with overwhelming probability, the distance between the $k$-th largest particle  and its classical position is of order $O(N^{-\frac{2}{3}+\epsilon}\min(k,N+1-k)^{-\frac{1}{3}})$.
\end{enumerate}
Here $\epsilon>0$ is an arbitrarily small constant. Our main idea is to decompose the multi-cut $\beta$ ensemble as a product of probability measures on spaces with lower dimensions and show that each of these measures is very close to a $\beta$ ensemble in one-cut regime for which the  rigidity of particles is known.
\end{abstract}

	 \section{Introduction}\label{chap:introduction}
	 \subsection{Background}\label{background}
	 A $\beta$ ensemble (or log-gas at inverse temperature $\beta$) is a probability measure $\mu=\mu(N)$ on
	 $\Sigma^{(N)}:=\{(x_1,\ldots,x_N)\in\mathbb R^N|x_1\le\cdots\le x_N\}$:
	 \begin{align}\label{density of the beta ensemble}
	 	\mu=\frac{1}{Z(\mu)}\exp\Big(-\frac{N\beta}{2}\sum_{i=1}^NV(x_i)\Big)\prod\limits_{i<j}|x_i-x_j|^\beta dx_1\cdots dx_N
	 \end{align}
	 where $Z(\mu)$ is the normalization constant and $dx_1\cdots dx_N$ is the Lebesgue measure. The constant $\beta>0$ is called $inverse$ $temperature$ and the function $V(x)$ is called $potential$. Each $x_i$ is called a $particle$ or an $eigenvalue$.
	 
	 For $\beta=1$ (resp. $\beta=2$, $\beta=4$), $\mu$ is the measure induced by the eigenvalues of a random symmetric (resp. Hermitian, quaternion Hermitian) $N$ by $N$ matrix $M_N$ with law $e^{-\frac{N\beta}{2}\text{tr}V(M_N)}dM_N$ where $dM_N$ is the Lebesgue measure on the set of $N$ by $N$ symmetric (resp. Hermitian, quaternion Hermitian) matrices. In case $V(x)$ is a quadratic polynomial, $\mu$ can be induced from a tri-diagonal random matrix model for all $\beta>0$ (see \cite{DE}), and its local limits can be described by  stochastic differential equations (see \cite{VV}). For $V(x)=\frac{x^2}{2}$ and $\beta=1$ (resp. $\beta=2$, $\beta=4$),  $\mu$ is the measure induced by the eigenvalues of a random matrix belonging to the classical Gaussian invariant GOE (resp. GUE, GSE) ensemble. The Gaussian invariant ensembles were  studied by Dyson, Gaudin and Mehta and the celebrated sine process was obtained in the bulk, see \cite{Mehta} for details.
	 
	 Consider the empirical measure of $\mu$:
	 \begin{align}\label{eq:284}
	 	L_N:=\frac{1}{N}\sum\limits_{i=1}^N\delta_{x_i}
	 \end{align}
	 where $\delta$ is the delta function. Obviously $L_N$ is a random probability measure on $\mathbb{R}$. It is well known that if $V(x)$ is continuous and increases faster than $2\ln|x|$ as $|x|\to+\infty$, then $L_N$
	 converges almost surely and in expectation to a compactly supported equilibrium measure $\mu_{eq}$ as $N$ tends to infinity. If $V$ is also real analytic, then $\mu_{eq}$ admits a continuous density function $\rho(t)$ and its support is  the union of a finite number of closed intervals (see, for example, Theorem 1.1 of \cite{BG_multi_cut}): 
	 $$\text{supp}(\mu_{eq})=\cup_{i=1}^q[A_i,B_i]\quad\quad(B_i<A_{i+1}).$$
	 If $q=1$, then we say that $\mu$ is in the one-cut regime. Otherwise we say that  $\mu$ is in the multi-cut regime and call each interval $[A_i,B_i]$ a $cut$.
	 
	 Bourgade,  Erd\H{o}s and Yau \cite{BEY_duke,BEY_bulk,BEY_edge} studied the $\beta$ ensemble in the one-cut regime. They proved the rigidity of eigenvalues in the bulk (see \cite{BEY_duke,BEY_bulk}) and also the rigidity of eigenvalues near the edges (see \cite{BEY_edge}). More precisely, if we  define the $k$-th quantile $\eta_k$ as in \eqref{eq:268}, then their theorem states that in the one-cut regime, for any $\epsilon>0$, with overwhelming $\mu$-probability, 
	 \begin{align}\label{eq:282}
	 	x_k-\eta_k
	 \end{align}
	 is of order $N^{-1+\epsilon}$ in the bulk of the spectrum and of order $N^{-\frac{2}{3}+\epsilon}$ near the edges of the spectrum. So the deterministic point $\eta_k$ is the $classical$ $position$ of the random variable $x_k$.

	 Bekerman \cite{Bekerman} proved that for $\beta$ ensembles in the multi-cut regime, the eigenvalues near $A_2$,\ldots,$A_q$, $B_1$,\ldots,$B_{q-1}$ can jump to the adjacent cut with a positive probability.  Therefore rigidity for eigenvalues near these edges does not hold.
	 
	 In this paper we will study the $\beta$ ensemble in the multi-cut regime and prove:
	 \begin{itemize}
	 	\item the rigidity of eigenvalues in the bulk of the spectrum;
	 	\item the rigidity of eigenvalues near the extreme edges (i.e., $A_1$ and $B_q$) of the spectrum.
	 \end{itemize}

	 Recently, Claeys, Fahs, Lambert and Webb \cite{Claeys+Fahs+Lambert+Webb} proved the following interesting result on rigidity for $\beta$ ensemble in the one-cut regime. If the support of the equilibrium measure is $[-1,1]$ and  $\beta=2$, then
	 $$\lim\limits_{N\to\infty}\mathbb P^\mu\Big(\frac{1-\epsilon}{\pi}\frac{\log N}{N}\le\max _{i=1,\ldots,N}\Big(p_V(\eta_i)\sqrt{1-\eta_i^2}|x_i-\eta_i|\Big)\le \frac{1+\epsilon}{\pi}\frac{\log N}{N}\Big)=1$$
	 where $\epsilon>0$ is an arbitrarily small constant and 
	 $p_V(x)$ is an $N$-independent deterministic function which never vanishes on $[-1,1]$. See Theorem 1.2 of \cite{Claeys+Fahs+Lambert+Webb}.
	 
	 If $x_1\le\cdots\le x_N$ are eigenvalues of a Wigner matrix, then it is well known that \eqref{eq:284} converges to the semicircle distribution. So we can define the classical locations and talk about the rigidity of the eigenvalues for a Wigner matrix in the same way. Moreover, the rigidity of eigenvalues is closely related to the  universality for $\beta$ ensemble as well as for Wigner matrix. The $global$ $universality$ states that, for a regular enough test function $f(x)$, the fluctuation of 
	 $$\sum_i f(x_i)$$ 
	 converges to a Gaussion distribution. For Wigner matrix, the global universality was first obtained  by Johansson \cite{Johansson} for GUE.   Bai and Yao \cite{Bai+Yao} proved the global universality for general Wigner matrix with analytic test function. Using the rigidity of eigenvalues, Sosoe and Wong \cite{Sosoe+Wong}  extended the result to $H^{1+\epsilon}$ test function. For $\beta$ ensemble, in the one-cut regime, Lambert, Ledoux and Webb \cite{Lambert+Ledoux+Webb} proved the global universality with the help of rigidity. For $\beta$ ensemble in the multi-cut regime, the global universality does not hold in general, but it was observed by M. Shcherbina \cite{Shcherbina_fluctuation} and proved by Bekerman, Lebl\'e and Serfaty \cite{Bekerman+Leble+Serfaty} that, if the test function $f(x)$ satisfies a certain system of equations, then the global universality holds.
	 
	 The $local$ $universality$ can be formulated in different ways, but they can all be understood as that  when $N$ goes to infinity, $N^r\cdot x_k$ behaves in a way independent of the distribution of the entries of the Wigner matrix  or the choice of the potential of the $\beta$ ensemble. Here $r=1$ in the bulk and $r=2/3$ at the edges.
	 The first significant result on local universality was obtained by Erd\H{o}s, Schlein and Yau for the symmetric Wigner matrix. See \cite{Erdos+Schlein+Yau3}. The most general result on local universality for generalized Wigner matrix was proved by Bourgade, Erd\H{o}s, Yau and Yin  and their proof uses the rigidity of eigenvalues. See \cite{BEYY}. For the $\beta$ ensemble, based on the rigidity of the eigenvalues, Bourgade,  Erd\H{o}s and Yau proved the local universality for $\beta$ ensemble in the one-cut regime both in the bulk of the spectrum (see \cite{BEY_duke,BEY_bulk}) and  at the edges of the spectrum (see \cite{BEY_edge}). Later, by the method of approximate transport maps, Bekerman, Figalli and Guionnet \cite{BFG} proved the local universality for $\beta$ ensemble in the one-cut regime, both in the bulk and at the edges. Also with the method of approximate transport maps, Bekerman \cite{Bekerman} proved the local university for $\beta$ ensemble in the multi-cut regime, both in the bulk and at the edges. By a change of variable method, M. Shcherbina \cite{Shcherbina_change_of_variables} proved the local universality for $\beta$ ensemble in the bulk, in both the one-cut regime and the multi-cut regime. However, the work of Bekerman, Figalli, Guionnet and Shcherbina do not need to use the rigidity of eigenvalues. Other works about universality for $\beta$ ensemble with the classical values $\beta\in\{1,2,4\}$ include \cite{DG1,DG2,PS,Shcherbina}.

	 Related topics on  $\beta$ ensembles include the fluctuation of linear statistics of eigenvalues (see \cite{Johansson} and Lemma 6.5 of \cite{BEY_edge} for the one-cut case; see  \cite{BG_multi_cut,Shcherbina_fluctuation} and Theorem \ref{thm:fluctuation_multi_cut} for the multi-cut case) and the asymptotic expansion of the correlators and partition function (see \cite{BG_one_cut,BG_multi_cut} and the reference therein). For GUE, Gustavsson \cite{Gustavsson} proved the central limit theorem for \eqref{eq:282} both in the bulk and near the edges of the spectrum.   Gustavsson's results were extended  to GOE and GSE by O'Rourke \cite{ORourke} and to Wigner matrix by Bourgade and Mody \cite{Bourgade+Mody} and Landon and Sosoe \cite{Landon+Sosoe}.
	 
	 The phenomenon of the rigidity of particles is observed in many other statistical models. It was first proved for the eigenvalues of Wigner matrices \cite{Erdos+Schlein+Yau,Erdos+Schlein+Yau2,Erdos+Yau+Yin,Erdos+Yau+Yin2,Tao+Vu}. Later it is  obtained  for the sparse random matrices \cite{Erdos+Knowles+Yau+Yin}, the deformed Wigner ensembles \cite{Landon+Yau,Lee+Schnelli+Stetler+Yau,Lee+Li}, the Dyson Brownian motion \cite{Huang+Landon}, some random graph models \cite{ADK,BHKY,BBK} and the discrete $\beta$ ensembles \cite{Guionnet+Huang}. In particular,  the discrete $\beta$ ensemble is an analogue of $\beta$ ensemble on  the space of $N$-tuples. Recently, significant results on the optimal rigidity of eigenvalues for the circular $\beta$-ensemble are obtained. See \cite{Arguin+Belius+Bourgade,Chhaibi+Madaule+Najnudel,Lambert2}.

	 \subsection{Main result}
	 
	 Let $\mu$ be a probability measure on
	 $$\Sigma^{(N)}=\{(x_1,\ldots,x_N)\in\mathbb R^N|x_1\le\cdots\le x_N\}$$
	 with density function
	 $$\frac{1}{Z(\mu)}e^{-\frac{N\beta}{2}\sum_{i=1}^NV(x_i)}\prod_{i<j}|x_i-x_j|^\beta$$
	 where $Z(\mu)$ is the normalization constant.
	 
	 The following lemma is well known. It first appeared (in a slightly different form) as Theorem 2.1 in \cite{Johansson}. See also  Theorem 1.1 of \cite{BG_one_cut}, Theorem
	 1.1 of \cite{BG_multi_cut}, Page 5 of \cite{Pastur} and Remark 1.9
	 of \cite{FGSW}. Related results on the equilibrium measure can be found in Chapter 6 of \cite{Deift}.

	 \begin{lemma}\label{lemma:minimizer}
	 	If $V$ is continuous and $\liminf\limits_{|x|\to\infty}\dfrac{V(x)}{\ln|x|}>2$, then:
	 	\begin{enumerate}
	 		\item $Z(\mu)<+\infty$ for large enough $N$; 
	 		\item The empirical measure
	 		$\frac{1}{N}\sum_{i=1}^N\delta_{x_i}$ converges almost surely and in expectation to a compactly supported equilibrium measure $\mu_{eq}$;
	 		\item The equilibrium
	 		measure $\mu_{eq}$ is the unique minimizer (in the set of
	 		probability measures on $\mathbb{R}$) of the functional:
	 		\begin{align*}
	 			\nu\mapsto\int V(x)d\nu(x)-\iint\ln|x-y|d\nu(x)d\nu(y);
	 		\end{align*}
	 		\item The equilibrium
	 		measure $\mu_{eq}$ satisfies:
	 		\begin{align}\label{eq:2255}
	 			V(x)-2\int_{\mathbb{R}}\ln|x-y|d\mu_{eq}(y)=\min\limits_{t\in\mathbb{R}}\Big[V(t)-2\int_{\mathbb{R}}\ln|t-y|d\mu_{eq}(y)\Big]\quad\text{for  $\mu_{eq}$-almost every $x$.}
	 		\end{align}
	 		On the other hand, if  $\nu$ is a  probability measure  such that \eqref{eq:2255} is satisfied with $\mu_{eq}$ replaced by $\nu$, then $\nu=\mu_{eq}$. 
	 		\item  If $V$ is real analytic, then
	 		\begin{itemize}
	 			\item $\mu_{eq}$ has a density $\rho(x)$ which is supported on the union of a finite number of disjoint intervals: $\text{supp}(\rho)=\cup_{j=1}^q[A_j,B_j]\quad(B_i<A_{i+1})$,
	 			\item $\rho(x)=r(x)\prod\limits_{i=1}^q\sqrt{x-A_i}\sqrt{x-B_i}$ on $\cup_{j=1}^q[A_j,B_j]$ and $r(x)$ is  analytic.
	 		\end{itemize}
	 		\item Suppose $V(x)-2\int_{\mathbb{R}}\ln|x-y|d\mu_{eq}(y)>\min\limits_{t\in\mathbb{R}}\Big[V(t)-2\int_{\mathbb{R}}\ln|t-y|d\mu_{eq}(y)\Big]$ for all $x\not\in\text{supp}(\mu_{eq})$. Let $A\subset\mathbb{R}$ be an open set containing 
	 		the support of $\mu_{eq}$. There exists $C>0$  such that for large enough $N$,
	 		\begin{align}\label{eq:2207}
	 			\P^{\mu}(\exists k\in[1,N]\text{ such that }x_k\not\in A)<\exp(-CN).
	 		\end{align}
	 		
	 	\end{enumerate}
	 \end{lemma}
	 \begin{remark}
	 	In many literature the $\beta$ ensemble is defined to be the probability measure on $I^N$:
	 	$$\mu'=\frac{1}{Z(\mu')}e^{-\frac{N\beta}{2}\sum_{i=1}^NV(x_i)}\prod_{i<j}|x_i-x_j|^\beta dx_1\cdots dx_N$$
	 	where $I$ is a finite or infinite interval and $Z(\mu')$ is the normalization constant. In this paper, to study the rigidity, we consider its variant $\mu''$ which is a probability measure defined on 
	 	$$\{(x_1,\ldots,x_N)\in I^N|x_1\le\cdots\le x_N\}$$
	 	with the form:
	 	$$\mu''=\frac{1}{ Z( \mu'')}e^{-\frac{N\beta}{2}\sum_{i=1}^NV(x_i)}\prod\limits_{i<j}|x_i-x_j|^\beta dx_1\cdots dx_N$$
	 	where $Z(\mu'')$ is the normalization constant. It is easy to see that $Z(\mu')=N!\cdot Z(\mu'')$ and
	 	\begin{align}\label{eq:2201}
	 		\mathbb E^{\mu'}[f(x_1,\ldots,x_N)]=\mathbb E^{\mu''}[f(x_1,\ldots,x_N)]
	 	\end{align}
	 	as long as $f(x_1,\ldots,x_N)$ is a symmetric function. For each result we cite, it may be proved for $\mu'$ in the original paper, but because of \eqref{eq:2201} it also holds for $\mu''$.
	 \end{remark}
	 
	 The following lemma is  Theorem 1.1 of \cite{BG_one_cut}.
	 \begin{lemma}\label{coro:convergence_for_restriction}
	 	Suppose $V$ and $\mu_{eq}$ are defined as in Lemma \ref{lemma:minimizer}. Suppose $I$ is a (either finite or infinite) closed interval with $\text{supp}(\mu_{eq})\subset \mathring I$. Suppose $\{V_N(x)\}$ is a sequence of continuous functions on $I$ such that $\sup\limits_{x\in I}|V_N(x)-V(x)|\to0$ as $N\to\infty$. Let $\bar\mu$ be a probability measure on
	 	$$\{(x_1,\ldots,x_N)\in I^N|x_1\le\cdots\le x_N\}$$
	 	with the form:
	 	$$\bar{\mu}=\frac{1}{ Z(\bar \mu)}\exp\Big(-\frac{N\beta}{2}\sum_{i=1}^NV_N(x_i)\Big)\prod\limits_{i<j}|x_i-x_j|^\beta\prod\limits_{i=1}^N\mathds{1}_I(x_i)dx_1\cdots dx_N$$
	 	where $Z(\bar{\mu})$ is the normalization constant. Then the empirical measure 	$\frac{1}{N}\sum_{i=1}^N\delta_{x_i}$ of $\bar{\mu}$ converges almost surely and in expectation to $\mu_{eq}$. In other words, the empirical measures of $\bar{\mu}$ and $\mu$ have the same limit.
	 \end{lemma}
	 
	 Throughout this paper we assume that the following Hypothesis are true.

	 \begin{hypothesis}\label{Hypothesis for initial model}
	 	\begin{enumerate}
	 		\item $V$ is real analytic. 
	 		\item $\liminf\limits_{|x|\to+\infty}\frac{ V(x)}{\ln|x|}>2$.
	 		\item $\inf\limits_{x\in\mathbb{R}}V''(x)>-\infty$.
	 		\item The equilibrium measure $\rho(x)dx$ has density
	 		\begin{align}\label{eq:239}
	 			\rho(x)=r(x)\Big(\prod\limits_{j=1}^q\sqrt{x-A_j}\sqrt{x-B_j}\Big)\mathds{1}_{\cup_{j=1}^q[A_j,B_j]}(x)\quad(B_i<A_{i+1})
	 		\end{align}
	 		and $r(x)$ does not vanish on $[A_i,B_i]$ for $1\le i\le q$. Here and throughout this paper we follow the convention that $\sqrt{-s}:={\rm i}\sqrt{|s|}$ for $s\in(-\infty,0)$.
	 		\item The function $x\mapsto
	 		V(x)-2\int_{\mathbb{R}}\ln|x-y|\rho(y)dy$ achieves its minimum only on the support of $\rho$:
	 		$$\sigma:=[A_1,B_1]\cup\ldots\cup[A_q,B_q].$$
	 	\end{enumerate}
	 \end{hypothesis}

	 \begin{remark}\label{remark: model and condition 4 are well defined}
	 	According to Lemma \ref{lemma:minimizer}, the model is well defined
	 	and Condition 4 in Hypothesis \ref{Hypothesis for initial model} makes sense.
	 \end{remark}
	 \begin{remark}
	 	We made the assumption that $\inf\limits_{x\in\mathbb{R}}V''(x)>-\infty$ in order to cite the theorem on the rigidity of $\beta$ ensemble in the one-cut regime proved by Bourgade, Erd\H{o}s and Yau. See Theorem \ref{thm:rigidity_for_mu_2}. They use this assumption to convexify the Hamiltonian $\frac{1}{2}\sum_i V(x_i)-\frac{1}{N}\sum_{i<j}\ln|x_i-x_j|$. See Section 3 of \cite{BEY_bulk}. We made the assumption that $V$ is analytic in order to cite M. Shcherbina's and Borot and Guionnet's theorems on $\E^\mu[e^{\text{linear statistics}}]$. See Theorem \ref{thm:Shcherbina_fluctuation_one_cut}, Theorem \ref{thm:fluctuation} and \eqref{eq:2233}. The other assumptions in Hypothesis \ref{Hypothesis for initial model} are generic.
	 \end{remark}

	 For $1\le k\le N$, define the classical location of the $k$th largest
	 particle $\eta_k=\eta^{(N)}_k$ by
	 \begin{align}\label{eq:268}
	 	\eta_k:=\inf\Big\{x\in\mathbb{R}\Big|\int_{-\infty}^x\rho(t)dt=\frac{k}{N}\Big\}.
	 \end{align}

	 For $1\le i\le q$,
	 define $R_i$ to be the area of the region under the curve of
	 $\rho(x)$ over $[A_i,B_i]$:
	 \begin{align}\label{eq:2206}
	 	R_i:=\int_{A_i}^{B_i}\rho(x)dx.
	 \end{align}
	 Obviously $\sum\limits_{i=1}^qR_i=1$. We make the convention that $R_0=0$. Our main results is Theorem
	 \ref{thm:main_thm}.

	 \begin{theorem}\label{thm:main_thm}
	 	Let $\alpha\in(0,\min_i\frac{R_i}{3})$ be an arbitrarily small constant. Suppose $\epsilon>0$ and $1\le i_0\le q$. There exists 
	 	$C>0$ such that
	 	if $N$ is large enough then
	 	$$
	 	\P^{\mu}(\exists k\in[(R_1+\cdots+R_{i_0-1}+\alpha)N,(R_1+\cdots+R_{i_0}-\alpha)N]\text{ such that }|x_k-\eta_k|>N^{-1+\epsilon})<\exp(-N^C),$$
	 	$$\mathbb P^{\mu}(\exists k\in[1,\alpha N]\cup[(1-\alpha)N,N]\text{ such that }|x_k-\eta_k|>N^{-\frac{2}{3}+\epsilon}\cdot\hat k^{-1/3})<\exp(-N^C).$$
	 	where $\hat k=\min(k,N+1-k)$.
	 \end{theorem}
	 
	 \begin{remark}
	 	The first conclusion gives the rigidity for particles in the bulk and the second conclusion gives the rigidity for particles near the extreme edges, i.e., the leftmost edge $A_1$ and the rightmost edge $B_q$. Theorem \ref{thm:main_thm} has been proved in the one-cut case (i.e., the case that $q=1$) by Bourgade,  Erd\H{o}s and Yau. (See Theorem 1.1 of \cite{BEY_bulk} or Theorem 2.4 of \cite{BEY_edge}.)
	 \end{remark}
	 \subsection{Strategy of the proof}
	 The inner structure of $\beta$ ensembles in the multi-cut regime and those in the one-cut regime are essentially different. The method Bourgade,  Erd\H{o}s and Yau used in \cite{BEY_duke,BEY_bulk,BEY_edge} to prove the rigidity for $\beta$ ensemble in the one-cut regime does not directly apply in the multi-cut regime for at least the following two reasons. 
	 
	 First, the starting point of the proof in \cite{BEY_duke,BEY_bulk,BEY_edge} is an $initial$ $estimation$ which states that if Hypothesis \ref{Hypothesis for initial model} is true and $q=1$, then for any $\epsilon>0$ there exists $C>0$ such that for large enough $N$ we have
	 \begin{align}\label{eq:283}
	 	\P^{\mu}(\exists k\in[1,N]\text{ such that }|x_k-\eta_k|>\epsilon)<\exp(-N^C).
	 \end{align}
	 But in the multi-cut regime, the particles near adjacent edges can jump with positive probability; see Section \ref{background}. So \eqref{eq:283} is no longer true.
	 
	 Second, in the multi-cut case, one cannot use the  loop equation to obtain a good estimation for 
	 $$|m_{N,h,c}(z)-m(z)|$$
	 as Lemma 6.6 of \cite{BEY_edge}. Here $m(z)$ denotes the Stieltjes transform of the empirical measure $\rho(t)dt$ and $m_{N,h,c}(z)$   denotes the Stieltjes transform of the empirical measure of the $\beta$ ensemble with density
	 $$\frac{1}{Z(h,c)}\exp\Big(-\frac{N\beta}{2}\sum_{i=1}^N V(x_i)+\beta\sum_{i=1}^N h(x_i)\Big)\cdot \prod_{i=1}^N\mathds1_{\cup_{j=1}^q[A_j-c,B_j+c]}(x_i)$$
	 where $Z(h,c)$ is the normalization constant, $c>0$ is a small constant and $h$ is a bounded continuous function. More explicitly, in the multi-cut case, the left hand side of (6.22) of \cite{BEY_edge} is the contour integral of
	 \begin{align}\label{eq:2216}
	 	\frac{(m_{N,h,c}(\xi)-m(\xi))\prod_{i=1}^q\sqrt{(\xi-A_i)(\xi-B_i)}}{z-\xi}
	 \end{align}
	 which is of order $|\xi|^{q-3}$ as $|\xi|\to\infty$. Therefore, in the multi-cut case, the left hand side of (6.22) of \cite{BEY_edge} is not necessarily the residue of \eqref{eq:2216} at $z$.  If we divide the integrands on both sides of (6.22) of \cite{BEY_edge} by $\prod_{i=1}^q\sqrt{(\xi-A_i)(\xi-B_i)}$, then the error terms in (6.23) of \cite{BEY_edge} will be out of control when $\text{Re}z\in\{A_i,B_i|1\le i\le q\}$. 
	 
	 The first difficulty mentioned above is essential, however, the second difficulty is  technical.
	 
	 We will follow the following steps  to prove the rigidity in the multi-cut regime.
	 \begin{enumerate}
	 	\item First, consider three $\beta$ ensembles all in the $one$-$cut$ $regime$. Let
	 	\begin{itemize}
	 		\item $\mu_1$ be a probability measure on $\{(x_1,\ldots,x_N)\in[a,b]^N|x_1\le\cdots\le x_N\}$ with density function:
	 		\begin{align}\label{eq:281}
	 			\frac{1}{Z(\mu_1)}e^{-\frac{N\beta}{2}\sum V_N(x_i)}\prod\limits_{i<j}|x_i-x_j|^\beta\prod_i\mathds1_{[a,b]}(x_i), 
	 		\end{align}
	 		\item $\mu_2$ and $\mu_3$ be   probability measures on $\Sigma^{(N)}$ with density functions:
	 		$$\frac{e^{-\frac{N\beta}{2}\sum V(x_i)}}{Z(\mu_2)}\prod\limits_{i<j}|x_i-x_j|^\beta\quad\text{and}\quad \frac{e^{-\frac{N\beta}{2}\sum V_N(x_i)}}{Z(\mu_3)}\prod\limits_{i<j}|x_i-x_j|^\beta$$
	 		respectively.
	 	\end{itemize}	
	 	Here $Z(\mu_1)$, $Z(\mu_2)$ and $Z(\mu_3)$ are the normalization constants. Here $\mu_2$ is the standard $\beta$ ensemble in the one-cut regime for which the bulk and edge rigidity  was proved in \cite{BEY_duke,BEY_bulk,BEY_edge}. The measure $\mu_3$ is constructed from $\mu_2$ by replacing $V$ by an $N$-depending potential $V_N$. The measure $\mu_1$ is constructed from $\mu_3$ by restricting all particles on in $[a,b]$ which is a neighbourhood of the support of the equilibrium measure. We will use the large deviation estimation \eqref{eq:2207} to show that $\mu_1$ is close to $\mu_3$ and use M. Shcherbina's expanssion of $\E[e^{\sum_i f(x_i)-N\int f(t)\rho(t)dt}]$ (i.e., Theorem \ref{thm:Shcherbina_fluctuation_one_cut}) to show that $\mu_3$ is close to $\mu_2$. So the rigidity of $\mu_2$ induces the rigidity of $\mu_1$. 
	 	\item In the $multi$-$cut$ $regime$, instead of the initial model \eqref{density of the beta ensemble}, we consider the measure $\mu_\kappa$ on $\{(x_1,\ldots,x_N)\in\sigma(\kappa)^N|x_a\le\cdots\le x_N\}$ with density
	 	$$\frac{1}{Z(\mu_\kappa)}\exp\Big(-\frac{N\beta}{2}\sum_{i=1}^NV(x_i)\Big)\prod\limits_{i<j}|x_i-x_j|^\beta\prod_i\mathds{1}_{\sigma(\kappa)}(x_i)$$
	 	where $Z(\mu_\kappa)$ is the normalization constant.	The measure $\mu_\kappa$ restricts all particles on $\sigma(\kappa)$ which is a neighbourhood of the support of the equilibrium measure. Because of the large deviation estimation \eqref{eq:2207}, $\mu_\kappa$ is close to $\mu$ and the rigidity of $\mu$ is induced from the rigidity of $\mu_\kappa$, both in the bulk and near the extreme edges.
	 	\item Construct a probability measure $\mu_r$  from $\mu_\kappa$ by changing the density function in the following ways:
	 	\begin{itemize}
	 		\item removing all terms in $\prod_{i<j}|x_i-x_j|^\beta$ which involve particles located in different cuts;
	 		\item replacing $V$ by a new potential $V^{(r)}$ to compensate the error caused by the removal.
	 	\end{itemize}
	 	The measure $\mu_r$ was introduced by M. Shcherbina \cite{Shcherbina}. We will prove in Proposition \ref{prop:small_for_mu_r_implies_small_for_mu_kappa} that $\mu_r$ is still close to $\mu_\kappa$ in the sense that if an event is exponentially small with respect to $\mu_r$ then it is also exponentially small with respect to $\mu_\kappa$. Therefore the rigidity of $\mu_r$ implies the rigidity of $\mu_\kappa$, both in the bulk and near the extreme edges.
	 	\item Since the intersection among particles in different cuts are removed, we prove that $\mu_r$ can basically be decomposed as a product of probability measures on spaces with lower dimensions. Moreover, each of these measures has the same form as $\mu_1$ (see \eqref{eq:281}), so the rigidity of $\mu_1$ implies the rigidity of $\mu_r$.
	 \end{enumerate}
	 We believe that using the same strategy one can generalize the existing results about the rigidity and universality for $\beta$ ensemble  in the following ways (as in \cite{BG_multi_cut} and \cite{BEY_edge}):
	 \begin{itemize}
	 	\item $V$ is $C^4$ instead of real analytic;
	 	\item $V$ depends on $N$ and converges to a limit $V^{\{0\}}$ uniformly;
	 	\item  $V$ is defined only on a neighborhood of $\cup_{i=1}^q[A_i,B_i]$ instead of  on all of $\mathbb{R}$.
	 \end{itemize}
	 
	 \subsection{Some potential applications of rigidity of eigenvalues}

	 \subsubsection{Mesoscopic universality for $\beta$ ensemble in the multi-cut regime}
	 
	 The $mesoscopic$ $universality$ states that, for some constant $a$ and $E$, the fluctuation of 
	 $$\sum f(N^a (x_i-E))$$
	 converges to a Gaussian distribution, when the test function $f(x)$ is regular enough. Here $a\in(0,1)$ if $E$ is in the bulk of the spectrum and $a\in(0,2/3)$ if $E$ is at the edges of the spectrum.  The mesoscopic universality can be interpreted as an intermediate phenomena between the global universality and the local universality.
	 
	 For Wigner matrix, the study of mesoscopic universality
	 was initiated by Boutet de Monvel and Khorunzhy \cite{BeoutetdeMonvel+Khorunzhy,BeoutetdeMonvel+Khorunzhy2}. In the bulk of the spectrum, He and Knowles \cite{He+Knowles} obtained the mesoscopic universality on the optimal scales. At the edges of the spectrum, the mesoscopic universality was studied by Basor and Widom \cite{Basor+Widom}  for GUE and by Min and Chen \cite{Min+Chen} for GOE and finally by Schnelli, Xu and the author \cite{Li+Schnelli+Xu} for the general case on the optimal scales.

	 For $\beta$ ensemble in the one-cut regime, Bekerman and Lodhia \cite{Bekerman+Lodhia} used rigidity to obtain the mesoscopic universality. In the multi-cut regime, for $\beta=2$, Lambert \cite{Lambert3} used the theory of determinantal point process to prove the mesoscopic universality. We remark that the results  \cite{Bekerman+Lodhia} and \cite{Lambert3} are both in the bulk of the spectrum. With the rigidity result proved in this paper, one may try to prove the mesoscopic universality for $\beta$ ensemble in the multi-cut regime for general $\beta>0$, both in the bulk and near the extreme edges.

	 \subsubsection{ Spherical Sherrington-Kirkpatrick model for $\beta$ ensemble in the multi-cut regime}
	 For a statistical model involving particles  $x_1\ge\cdots\ge x_N$, the free energy of its
	 spherical Sherrington–Kirkpatrick model (with 2-spin interaction and no magnetic
	 field) at inverse temperature $\beta_0>0$ is given by 
	 $$F_N=\frac{1}{N}\log\Bigg(\frac{\Gamma(N/2)}{2\pi{\rm i}(N\beta_0)^{\frac{N}{2}-1}}\int_{a_0-{\rm i}\infty}^{a_0+{\rm i}\infty}\exp\Big(N\beta_0 z-\frac{1}{2}\sum_{i=1}^N\log(z-x_i)\Big)dz\Bigg)$$
	 where 
	 \begin{itemize}
	 	\item $a_0$ is an arbitrary constant satisfying $a_0>x_1$;
	 	\item we take the analytic branch of the $\log$ function in the integral such that $\text{Im}\log(z-x_i)\in(-\pi,\pi)$ for all $z$ on the integration contour;
	 	\item the value of the integral can be proved to be in ${\rm i}\mathbb R$, so the outer $\log$ is the usual $\log$ function for real numbers.
	 \end{itemize}
	 See (1.2) and Lemma 1.3 of \cite{Baik+Lee}. For the Wigner matrix, the sample covariance matrix and $\beta$ ensemble in the one-cut regime, Baik and Lee \cite{Baik+Lee} proved the following result. There exists $\beta_c>0$ such that if $\beta_0<\beta_c$, then the fluctuation of $F_N$ converges to a Gaussian distribution; if $\beta_0>\beta_c$, then  the fluctuation of $F_N$ converges to the Tracy-Widom distribution. The main tool they used for the case $\beta_0<\beta_c$ is the global universality and the main tools used for the case $\beta_0>\beta_c$ are the rigidity and the fact that $x_1$ converges to the Tracy-Widom distribution. One may consider the same problem for $\beta$ ensemble in the multi-cut regime. In the multi-cut regime, the main difficulty for the case $\beta_0<\beta_c$ is that  the global universality does not hold for general test function, as we mentioned in Section \ref{background}. For the case $\beta_0>\beta_c$, there are two main difficulties: (i)  rigidity does not hold near the edges $A_2$,\ldots, $A_q$, $B_1$,\ldots, $B_{q-1}$; (ii) the the limiting behavior of $x_1$ is not known in the multi-cut regime. Notice that $\beta$ and $\beta_0$ are two independent constants.

	 \subsection{Structure of this paper} 
	 In Section \ref{sec:one_cut} we consider $\beta$ ensembles in the one-cut regime. In Section \ref{some_models} we define three $\beta$ ensembles  $\mu_1$, $\mu_2$ and $\mu_3$ which are all in the one-cut regime. In Section \ref{sec:results_for_mu_2} we introduce some useful results for   $\mu_2$. In Section \ref{sec:rigidity_of_mu_1} we use the properties of $\mu_2$ to prove the rigidity of $\mu_1$.

	 In Section \ref{sec:decomposition of beta ensemble} we decompose the $\beta$ ensemble in the multi-cut regime as a product of $\beta$ ensembles in the one-cut regime. In Section \ref{sec:mu_kappa} we define the measure $\mu_\kappa$ which is constructed from $\mu$ by restricting all particles on a compact set. In Section \ref{sec:mu_r} we define $\mu_r$ which is constructed from $\mu_\kappa$ by removing the intersection among particles in different cuts and modifying the potential accordingly. In Section \ref{sec:the measure nu(i,N,cN)} we define the measure  $\nu^{(i,M,\xi)}$  which has the same form as $\mu_1$ and is also a factor of $\mu_r$. In other words we decompose $\mu_r$ as a product of some measures which have the same form as $\mu_1$.
	 
	 In Section \ref{sec:proof of main thm} we prove the main theorem.

	 In Section \ref{appendix:proof_of_fluctuation_for_initial_model} we prove a large deviation estimation for the fluctuation of linear statistics of eigenvalues.

	 \section{Some results for $\beta$ ensembles in the one-cut regime}\label{sec:one_cut} 
	 
	 \subsection{Some $\beta$ ensembles in the one-cut regime}\label{some_models}
	 
	 To prove the main result, we need some results for $\beta$ ensembles in the one-cut regime.
	 
	 \begin{definition}\label{def:V_0}
	 	Suppose $[c,d]\subset[a,b]$ with $c-a=b-d>0$. Suppose $D$ is a domain in $\mathbb{C}$ and $[a,b]\subset D$. Suppose $V_0(x):\mathbb{R}\to\mathbb{R}$ and $r_0(x):[c,d]\to(0,+\infty)$ satisfy the following conditions.
	 	\begin{itemize}
	 		\item $V_0\in C^\infty(\mathbb{R})$ and $\liminf\limits_{|x|\to+\infty}\frac{V_0(x)}{\ln|x|}>2$. Moreover, $V_0\big|_{D\cap\mathbb{R}}$ can be analytically extended to $D$.
	 		\item $\inf\limits_{x\in\mathbb{R}} V_0''(x)>-\infty$.
	 		\item $r_0(x)$ can be analytically extended to $D$. Moreover, $r_0(z)\ne0$ for $z\in D$. 
	 		\item The function $\rho_0(x):=r_0(x)\sqrt{(x-c)(d-x)}\mathds{1}_{[c,d]}(x)$ satisfies the condition that $\int_c^d\rho_0(x)dx=1$ and that
	 		\begin{align*}
	 			V_0(x)-2\int_c^d\rho_0(y)\ln|x-y|dy\begin{cases}
	 				=\min\limits_{x\in\mathbb{R}}\big(V_0(x)-2\int_c^d\rho_0(y)\ln|x-y|dy\big)\,&\text{if }x\in[c,d]\\>\min\limits_{x\in\mathbb{R}}\big(V_0(x)-2\int_c^d\rho_0(y)\ln|x-y|dy\big)\,&\text{if }x\in\mathbb R\backslash[c,d]
	 			\end{cases}
	 		\end{align*}
	 	\end{itemize}
	 \end{definition}
	 
	 \begin{definition}\label{def:measures} Suppose $\{c_N\}$ is a sequence of numbers such that $|c_N-1|<N^{-1+\epsilon_0}$ for some constant $\epsilon_0\in(0,0.01)$. Suppose $\tau>0$ is a constant  such that $[a-\tau,b+\tau]\subset D$. Let $\phi(x):\mathbb{R}\to[0,1]$ be a smooth function such that $\phi(x)=1$ if $x\in[a,b]$ and $\phi(x)=0$ if $x\not\in[a-\tau,b+\tau]$. Now let:
	 	
	 	\begin{itemize}
	 		\item  $\mu_1=\mu_1(N)$ be a probability measure on $\Sigma_{[a,b]}^{(N)}:=\{(x_1,\ldots,x_N)\in[a,b]^N|x_1\le\cdots\le x_N\}$ with density
	 		\begin{align*}
	 			\frac{1}{Z(\mu_1)}e^{-\frac{N\beta}{2}\sum_{i=1}^Nc_NV_0(x_i)}\prod\limits_{i<j}|x_i-x_j|^\beta\prod_{i=1}^N\mathds{1}_{[a,b]}(x_i)
	 		\end{align*}
	 		where $Z(\mu_1)$ is the normalization constant;
	 		\item  $\mu_2=\mu_2(N)$ be a probability measure on $\Sigma^{(N)}$ with density
	 		\begin{align*}
	 			\frac{1}{Z(\mu_2)}e^{-\frac{N\beta}{2}\sum_{i=1}^NV_0(x_i)}\prod\limits_{i<j}|x_i-x_j|^\beta
	 		\end{align*}
	 		where $Z(\mu_2)$ is the normalization constant;
	 		\item  $\mu_3=\mu_3(N)$ be a probability measure on $\Sigma^{(N)}$ with density
	 		\begin{align*}
	 			\frac{1}{Z(\mu_3)}e^{-\frac{N\beta}{2}\sum_{i=1}^NV_N(x_i)}\prod\limits_{i<j}|x_i-x_j|^\beta
	 		\end{align*}
	 		where $V_N(x)=V_0(x)(1+(c_N-1)\phi(x))$ and $Z(\mu_3)$ is the normalization constant. 
	 	\end{itemize}
	 \end{definition}
	 
	 \begin{lemma}\label{lemma:limiting_measure}
	 	Suppose $f(x)$ is a continuous and bounded function on $\mathbb{R}$. Then $\frac{1}{N}\E^{\mu_1}[\sum_{i=1}^Nf(x_i)]$, $\frac{1}{N}\E^{\mu_2}[\sum_{i=1}^Nf(x_i)]$ and $\frac{1}{N}\E^{\mu_3}[\sum_{i=1}^Nf(x_i)]$ all converge to $\int f(x)\rho_0(x)dx$ as $N\to\infty$.  In other words, the empirical measures of $\mu_1$, $\mu_2$ and $\mu_3$ have the same limit.
	 \end{lemma}
	 \begin{proof}
	 	By  Lemma \ref{lemma:minimizer} we have the convergence of $\frac{1}{N}\E^{\mu_2}[\sum_{i=1}^Nf(x_i)]$.  By Lemma \ref{coro:convergence_for_restriction} we have the convergence of $\frac{1}{N}\E^{\mu_1}[\sum_{i=1}^Nf(x_i)]$ and $\frac{1}{N}\E^{\mu_3}[\sum_{i=1}^Nf(x_i)]$.
	 \end{proof}

	 \begin{definition}
	 	For $1\le k\le N$, define the $k$th classical location $\eta_k^0=\eta_k^0(N)$ by
	 	\begin{align}\label{eq:251}
	 		\eta_k^0=\inf\Big\{x\in\mathbb{R}\Big|\int_{-\infty}^x\rho_0(x)dx=\frac{k}{N}\Big\}.
	 	\end{align}
	 \end{definition}
	 
	 \subsection{Some results for $\mu_2$}\label{sec:results_for_mu_2}
	 
	 The next theorem on $\beta$ ensemble in the one-cut regime was proved by Bourgade,  Erd\H{o}s and Yau. (See Theorem 2.4 of \cite{BEY_edge}.)
	 \begin{theorem}\label{thm:rigidity_for_mu_2}
	 	For any  $\epsilon>0$, there exist 
	 	$C>0$  such that
	 	\begin{align*}
	 		\mathbb P^{\mu_2}(\exists k\in[1,N]\text{ such that }|x_k-\eta_k^0|>N^{-\frac{2}{3}+\epsilon}\cdot\hat k^{-1/3})<\exp(-N^C)
	 	\end{align*}
	 	for large enough $N$. Here   $\hat k=\min(k,N+1-k)$, as defined in Theorem \ref{thm:main_thm}.
	 \end{theorem}
	 
	 The next theorem was proved by M. Shcherbina. (See Theorem 1 of \cite{Shcherbina_fluctuation}.) 
	 \begin{theorem}\label{thm:Shcherbina_fluctuation_one_cut}
	 	Suppose $\{h_N:\mathbb{R}\to\mathbb{R}|N\in\mathbb{N}\}$ is a sequence of smooth functions all supported on $[a-\tau,b+\tau]$ such that $\max(\|h_N'\|_\infty,\|h_N^{(6)}\|_\infty)\le\sqrt N\ln N$. Then for $N\ge1$,
	 	\begin{align}\label{eq:231}
	 		\mathbb{E}^{\mu_2}\Big[e^{\frac{\beta}{2}\sum_{i=1}^N h_N(x_i)-\frac{\beta}{2}N\int\rho_0(t)h_N(t)dt}\Big]
	 		=\exp\Bigg[\sigma^2(h_N)+2\int_c^dh_N(x)d\nu(x)+\Phi(h_N)\Bigg]
	 	\end{align}
	 	where 
	 	\begin{itemize}
	 		\item 
	 		\begin{align}\label{eq:2215}
	 			\sigma^2(h_N)=\frac{\beta}{8\pi^2}\int_c^d\frac{h_N(x)}{\sqrt{(x-c)(d-x)}}\text{P.V.}\int_c^d\frac{h_N'(y)\sqrt{(y-c)(d-y)}}{x-y}dydx
	 		\end{align}
	 		
	 		\item $\nu$ is a finite signed measure on $[c,d]$ and $\nu$ is determined by $V_0$;
	 		\item $|\Phi(h_N)|\le\frac{C}{N}(\|h_N'\|_\infty^3+\|h_N^{(6)}\|_\infty^3)$ where $C>0$ is a constant.
	 	\end{itemize}
	 \end{theorem}
	 \begin{remark}
	 	According to (4.16) of \cite{Lambert}, 
	 	\begin{align}\label{eq:2213}
	 		|\sigma^2(h_N)|\le C_1\cdot\|h_N'\|_\infty^2.
	 	\end{align} 
	 	Here $C_1>0$ is a constant. Notice that (4.16) of \cite{Lambert} requires the support of the equilibrium measure to be $[-1,1]$, but this can be satisfied by a linear translation: $V_0(x)\mapsto V_0(\frac{d-c}{2}x+\frac{d^2-c^2}{4})$. One can also prove \eqref{eq:2213} by direct computation.
	 \end{remark}
	 \begin{remark}
	 	\begin{enumerate}
	 		\item Instead of  Johansson's loop equation method (see \cite{Johansson}), the main idea used in  \cite{Shcherbina_fluctuation}  to prove Theorem \ref{thm:Shcherbina_fluctuation_one_cut} is to control the quantity
	 		$$N\int_{\mathbb R}(\rho_{N,h}(t)-\rho_0(t))\varphi(t)dt$$
	 		where 
	 		\begin{itemize}
	 			\item $\rho_{N,h}(t)$ is the one-point correlation function of the $\beta$ ensemble with potential  $V_0(x)-\frac{1}{N}h_N(x)$;
	 			\item  $\varphi(t)$ is an arbitrary function with bounded sixth derivatives.
	 		\end{itemize}
	 		\item  Theorem \ref{thm:Shcherbina_fluctuation_one_cut} does not show that the error term $\Phi(h_N)$ is small compared to the first two terms in the exponent of \eqref{eq:231}. For example, if $\|h_N^{(6)}\|_\infty$ is close to $\sqrt N\ln N$ and $\|h_N'\|_\infty\approx\|h_N\|_\infty\approx 0$, then $\frac{C}{N}(\|h_N'\|_\infty^3+\|h_N^{(6)}\|_\infty^3)$, i.e.,  the bound of $\Phi(h_N)$ given by Theorem \ref{thm:Shcherbina_fluctuation_one_cut}, will be much larger than the other two terms in the exponent of \eqref{eq:231}. However, when we apply Theorem \ref{thm:Shcherbina_fluctuation_one_cut} in this paper, we actually have $\|h_N^{(i)}\|_\infty\le C_1N^{\epsilon_0}$ ($i=0,\ldots,6$) where $C_1>0$ and $\epsilon_0\in(0,0.01)$ are constants. Since $\epsilon_0$ is small,   $\Phi(h_N)$  is really much smaller than the natural bounds of the other terms in the exponent of \eqref{eq:231}. See the proof of Lemma \ref{lemma:mu_1mu_2mu_3}.
	 	\end{enumerate}
	 \end{remark}

	 \subsection{Rigidity for $\mu_1$}\label{sec:rigidity_of_mu_1}
	 
	 In this section we prove the rigidity of particles with respect to $\mu_1$.
	 
	 \begin{lemma}\label{lemma:trivial}
	 	Suppose $m_1$ and $m_2$ are probability measures defined on a same probability space such that
	 	$$m_1=\frac{f\cdot m_2}{\E^{m_2}[f]}$$
	 	where $f$ is a measurable function integrable with respect to $m_2$. Then for any event $A$ we have
	 	$$\P^{m_1}(A)=\frac{\E^{m_2}[f\cdot\mathds1_A]}{\E^{m_2}[f]}$$
	 \end{lemma}
	 \begin{proof}
	 	This lemma is trivial.
	 \end{proof}
	 
	 \begin{lemma}\label{lemma:mu_1mu_2mu_3}
	 	Suppose $A_N$ is a Borel subset of $\Sigma^{(N)}$. There exist constants $C>0$ and $N_0>0$ such that if $N>N_0$ then
	 	\begin{align*}
	 		\P^{\mu_3}(A_N)\le\sqrt{\P^{\mu_2}(A_N)}\exp(CN^{2\epsilon_0})\quad\text{and}\quad \P^{\mu_1}(A_N)\le2\P^{\mu_3}(A_N)
	 	\end{align*}
	 	where $\Sigma_{[a,b]}^{(N)}=\{(x_1,\ldots,x_N)\in[a,b]^N|x_1\le\cdots\le x_N\}$, as defined in Definition \ref{def:measures}.
	 \end{lemma}
	 \begin{proof}
	 	Let $$h_N(x):=N(V_N(x)-V_0(x))=NV_0(x)(c_N-1)\phi(x)\in C^\infty(\mathbb{R}).$$
	 	By the definition of $c_N$ and $h_N$, there exist constants $N_0>0$ and $C_2>0$ such that if $N>N_0$ then 
	 	\begin{align}\label{eq:2214}
	 		\|h_N'\|_\infty+\|h_N^{(6)}\|_\infty<C_2N^{\epsilon_0}< \frac{1}{2}\sqrt N\ln N
	 	\end{align}
	 	and thus by Lemma \ref{lemma:trivial} and \eqref{eq:231} we have:
	 	\begin{align}\label{eq:259}
	 		&\P^{\mu_3}(A_N)
	 		=\dfrac{\E^{\mu_2}\big[e^{-\frac{\beta}{2}\sum h_N(x_i)+\frac{N\beta}{2}\int h_N(t)\rho_0(t)dt}\cdot\mathds{1}_{A_N}(x)\big]}{\E^{\mu_2}\big[e^{-\frac{\beta}{2}\sum h_N(x_i)+\frac{N\beta}{2}\int h_N(t)\rho_0(t)dt}\big]}\nonumber\\
	 		\le&\dfrac{\sqrt{\P^{\mu_2}(A_N)}\sqrt{\mathbb E^{\mu_2}\big[e^{-\beta\sum h_N(x_i)+N\beta\int h_N(t)\rho_0(t)dt}\big]}}{\E^{\mu_2}\big[e^{-\frac{\beta}{2}\sum h_N(x_i)+\frac{N\beta}{2}\int h_N(t)\rho_0(t)dt}\big]}\quad\text{(by Cauchy-Schwarz inequality)}\nonumber\\
	 		=&\dfrac{\sqrt{\P^{\mu_2}(A_N)}\sqrt{\exp\Big(\sigma^2(-2h_N)-4\int_c^dh_N(x)d\nu(x)+\Phi(-2h_N)\Big)}}{\exp\Big(\sigma^2(-h_N)-2\int_c^dh_N(x)d\nu(x)+\Phi(-h_N)\Big)}\nonumber\\
	 		=&\sqrt{\P^{\mu_2}(A_N)}\exp\Big(\sigma^2(h_N)+\frac{\Phi(-2h_N)}{2}-\Phi(-h_N)\Big)\quad\text{(since $\sigma^2(-2h_N)=4\sigma^2(h_N)=4\sigma^2(-h_N)$)}\nonumber\\
	 		\le&\sqrt{\P^{\mu_2}(A_N)}\exp\Big( C_0N^{2\epsilon_0}+\frac{\Phi(-2h_N)}{2}-\Phi(-h_N)\Big)\quad\text{(by \eqref{eq:2213} and \eqref{eq:2214})}\nonumber\\
	 		\le&\sqrt{\P^{\mu_2}(A_N)}\exp(C_0N^{2\epsilon_0}+2C_1)
	 	\end{align}
	 	where
	 	\begin{itemize}
	 		\item $C_0>0$ is a constant;
	 		\item $\max(|\Phi(-2h_N)|,|\Phi(-h_N)|)\le C_3N^{-1+3\epsilon_0}$ for some constant $C_3>0$.
	 	\end{itemize}
	 	This proves the first inequality. For the second inequality, By definition $V_N(x)=c_NV_0(x)$ if $x\in[a,b]$ and $V_N(x)=V_0(x)$ if $x\not\in[a-\tau,b+\tau]$. By Lemma \ref{lemma:trivial},
	 	\begin{align}\label{eq:260}
	 		\P^{\mu_1}(A_N)
	 		=&\big(\P^{\mu_3}(\Sigma_{[a,b]}^{(N)})\big)^{-1}\P^{\mu_3}(A_N\cap\Sigma_{[a,b]}^{(N)})\nonumber\\
	 		\le&\big(\P^{\mu_3}(\Sigma_{[a,b]}^{(N)})\big)^{-1}\P^{\mu_3}(A_N).
	 	\end{align}
	 	
	 	Let $B_N=\{(x_1,\ldots,x_N)\in\Sigma^{(N)}|\exists x_i\not\in[a,b]\}$. By the first  conclusion of this lemma, for $N>N_0$ we have
	 	\begin{align}\label{eq:262}
	 		&1-\P^{\mu_3}(\Sigma_{[a,b]}^{(N)})=\P^{\mu_3}(B_N)\le\sqrt{\P^{\mu_2}(B_N)}\exp(CN^{2\epsilon_0}).
	 	\end{align}
	 	According to the last conclusion of Lemma \ref{lemma:minimizer}, there are constants $C_4>0$ and $N_0>0$ such that if $N>N_0$, then 
	 	\begin{align}\label{eq:263}
	 		\P^{\mu_2}(B_N)<\exp(-C_4N).
	 	\end{align}
	 	According to \eqref{eq:262} and \eqref{eq:263}, if $N>N_0$, then 
	 	\begin{align}\label{eq:264}
	 		1-\P^{\mu_3}(\Sigma_{[a,b]}^{(N)})<\frac{1}{2}.
	 	\end{align}
	 	\eqref{eq:260} and \eqref{eq:264} prove the second inequality.
	 \end{proof}
	 
	 \begin{coro}\label{thm:rigidity_for_mu_1}
	 	For any $\epsilon>0$, there exist constants $N_0>0$ and
	 	$C>0$  such that
	 	if $N>N_0$ and $\epsilon_0<C/10$, then
	 	$$\P^{\mu_1}(\exists k\in[1,N]\text{ s.t. }|x_k-\eta_k^0|>N^{-\frac{2}{3}+\epsilon} \cdot\hat k^{-1/3})<\exp(-N^{C/2}).$$
	 	Here $\eta_k^0$ and $\epsilon_0$ are defined in Section \ref{some_models}.
	 \end{coro}
	 
	 \begin{proof} Set
	 	$$A_N=\{(x_1,\ldots,x_N)\in\Sigma^{(N)}\big|\exists k\in[1,N]\text{ such that }|x_k-\eta_k^0|>N^{-\frac{2}{3}+\epsilon}\cdot\hat k^{-1/3}\}.$$
	 	According to Theorem \ref{thm:rigidity_for_mu_2} there are $C_1>0$ and $N_0>0$  such that if $N>N_0$ then
	 	\begin{align}\label{eq:241}
	 		\P^{\mu_2}(A_N)\le\exp(-N^{C_1})
	 	\end{align}
	 	According to \eqref{eq:241} and Lemma \ref{lemma:mu_1mu_2mu_3}, there exist $C_2>0$ and  $N_0>0$ such that if $N>N_0$ and $\epsilon_0<C_1/10$, then
	 	\begin{align}\label{eq:236}
	 		\P^{\mu_1}(A_N)\le 2\sqrt{\P^{\mu_2}(A_N)}\exp(C_2N^{2\epsilon_0})\le2\exp(-\frac{1}{2}N^{C_1}+C_2N^{0.2C_1})\le\exp(-N^{C_1/2}).
	 	\end{align}
	 \end{proof}

	 \section{Decomposition of beta ensemble in the multi-cut regime}\label{sec:decomposition of beta ensemble}

	 Notice that we can rewrite the density of $\mu$ as
	 $$\frac{1}{Z(\mu)}e^{-\frac{N\beta}{2}\sum_{i=1}^NV(x_i)}\prod_{i<j}|x_i-x_j|^\beta=\frac{1}{Z(\mu)}e^{-\beta N\mathcal H(x)}$$
	 where 
	 \begin{align}\label{eq:2202}
	 	\mathcal H(x)=\frac{1}{2}\sum_{i=1}^NV(x_i)-\frac{1}{2N}\sum_{i\ne j}\ln|x_i-x_j|.
	 \end{align}
	 Recall that the support of the equilibrium measure of $\mu$ is $\sigma=[A_1,B_1]\cup\ldots\cup[A_q,B_q]$.
	 
	 \begin{definition}
	 	Suppose $\kappa>0$ is a small  constant  satisfying
	 	the following conditions:
	 	\begin{enumerate}\label{condition_on_kappa}
	 		\item $\kappa<\frac{1}{100}\min(A_2-B_1,A_3-B_2,\ldots,A_q-B_{q-1})$   and $\kappa<0.1$,
	 		\item $V(x)$ and $r(x)$ can be analytically extended to a neighborhood of
	 		$\{z\in\mathbb{C}|\text{dist}(z,\sigma)\le30\kappa\}$,
	 		\item $r(z)$ does not vanish on a neighborhood of
	 		$\{z\in\mathbb{C}|\text{dist}(z,\sigma)\le30\kappa\}$.
	 	\end{enumerate}
	 	
	 \end{definition}
	 
	 \begin{remark}
	 	The $\kappa$ satisfying the above conditions exists because of the fifth conclusion in Lemma \ref{lemma:minimizer} and  Hypothesis \ref{Hypothesis for initial model}. 
	 \end{remark}
	 \begin{definition}\label{def:sets}
	 	\begin{itemize}
	 		\item For $1\le i\le q$, let $\sigma_i=[A_i,B_i]$ and
	 		$\sigma_i(\kappa)=[A_i-\frac{\kappa}{2},B_i+\frac{\kappa}{2}]$.
	 		\item   Set
	 		$\sigma(\kappa)=\cup_{i=1}^q\sigma_i(\kappa)$ and
	 		$\Sigma_\kappa^{(N)}=\{x\in\sigma(\kappa)^N|x_1\le\ldots\le x_N\}$.
	 	\end{itemize}
	 \end{definition}
	 Obviously 
	 $\cup_{i=1}^q\sigma_i=\sigma$ and
	 $\sigma_i(\kappa)\cap\sigma_j(\kappa)=\emptyset$ if $i\ne j$.

	 \begin{theorem}\label{thm:fluctuation_multi_cut}
	 	Suppose $O\subset\mathbb C$ is a neighborhood of $\cup_{i=1}^q[A_i,B_i]$. Suppose $h_N(x)\in  C^\infty(\mathbb{R})$ and there is a constant $C_b>0$ such that
	 	\begin{itemize}
	 		\item  $h_N$ can be analytically extended to $O$ and $\sup_{z\in O}|h_N(z)|<C_b$ for all $N$;
	 		\item $\|h_N^{(i)}\|_\infty<C_b$ for $0\le i\le 6$ and $N\ge1$. Here $\|h_N^{(i)}\|_\infty=\sup_{x\in\mathbb R}|h_N^{(i)}(x)|$.
	 	\end{itemize}
	 	Then for any $w_1>0$ there exists $C>0$ and $N_0>0$  such that if $N>N_0$ then
	 	\begin{align*}
	 		\P^\mu\Big(\Big|\sum\limits_{i=1}^Nh_N(x_i)-N\int h_N(t)\rho(t)dt\Big|>N^{w_1}\Big)<\exp(-N^C).
	 	\end{align*}
	 \end{theorem}
	 Theorem
	 \ref{thm:fluctuation_multi_cut} was essentially proved in \cite{Shcherbina_fluctuation} and \cite{BG_multi_cut}, but was expressed in a different way there. For the convenience of readers, we provide a proof of Theorem
	 \ref{thm:fluctuation_multi_cut} in Section
	 \ref{appendix:proof_of_fluctuation_for_initial_model}.

	 \subsection{The measure $\mu_\kappa$ on $\Sigma_\kappa^{(N)}$}\label{sec:mu_kappa}
	 
	 Suppose $\mu_\kappa=\mu_\kappa^{(N)}$ is a probability measure on
	 $\Sigma_\kappa^{(N)}$ with density
	 
	 $$\frac{1}{Z(\mu_\kappa)}\exp(-\beta N\mathcal H(x))\prod\limits_{i=1}^N\mathds{1}_{\sigma(\kappa)}(x_i)$$
	 where $Z(\mu_\kappa)$ is the normalization function. So $\mu_\kappa$
	 depends on $V$, $\kappa$ and $N$.
	 
	 \begin{lemma}\label{lemma:fluctuation_for_mu_kappa}
	 	Suppose the assumptions of Theorem \ref{thm:fluctuation_multi_cut} hold. Then for any $\tau>0$, there exist constants $C>0$, $N_0>0$ such that if
	 	$N>N_0$, then
	 	$$\P^{\mu_\kappa}\Big(\Big|\sum\limits_{i=1}^Nh_N(x_i)-N\int_{\mathbb{R}}h_N(t)\rho(t)dt\Big|>N^{\tau}\Big)<\exp(-N^C).$$
	 \end{lemma} 
	 \begin{proof} Suppose $x=(x_1,\ldots,x_N)$.
	 	According to Lemma \ref{lemma:trivial},
	 	\begin{multline*}
	 		\P^{\mu_\kappa}\Big(\Big|\sum\limits_{i=1}^Nh_N(x_i)-N\int_{\mathbb{R}}h_N(x)\rho(x)dx\Big|>N^{\tau}\Big)\\
	 		=\Big(\P^\mu\Big(\Sigma_\kappa^{(N)}\Big)\Big)^{-1}\P^\mu\Big(\Big|\sum\limits_{i=1}^Nh_N(x_i)-N\int_{\mathbb{R}}h_N(x)\rho(x)dx\Big|>N^{\tau}\text{
	 			and
	 		}x\in\Sigma_\kappa^{(N)}\Big)
	 	\end{multline*}
	 	According to the last conclusion of Lemma \ref{lemma:minimizer},  there exist
	 	$C_1>0$ and $N_1>0$ such that if
	 	$N>N_1$, then
	 	$\P^\mu\Big(\Sigma_\kappa^{(N)}\Big)>1-\exp(-C_1N)>\frac{1}{2}$
	 	and therefore
	 	\begin{multline*}
	 		\P^{\mu_\kappa}\Big(\Big|\sum\limits_{i=1}^Nh_N(x_i)-N\int_{\mathbb{R}}h_N(x)\rho(x)dx\Big|>N^{\tau}\Big)
	 		\le2\P^\mu\Big(\Big|\sum\limits_{i=1}^Nh_N(x_i)-N\int_{\mathbb{R}}h_N(x)\rho(x)dx\Big|>N^{\tau}\Big).
	 	\end{multline*}
	 	Applying Theorem \ref{thm:fluctuation_multi_cut}  we complete the
	 	proof.
	 \end{proof}

	 \begin{coro}\label{lemma:initial estimate 3}
	 	Suppose $i\in\{1,\ldots,q\}$ and $\tau>0$. Set
	 	$$\Omega_N^{(i)}(\tau)=\{|\sharp\{j|x_j\in \sigma_i(\kappa)\}-NR_i|\le N^\tau\}$$
	 	and
	 	\begin{align}\label{eq:2252}
	 		\Omega_N(\tau)=\cap_{i=1}^q\Omega_N^{(i)}(\tau).
	 	\end{align}
	 	There exist constants $C>0$,
	 	$N_0>0$   such that if
	 	$N>N_0$, then
	 	$$\P^{\mu_\kappa}\Big(\Big(\Omega_N^{(i)}(\tau)\Big)^c\Big)<\exp(-N^C),\quad \P^{\mu_\kappa}\Big(\Big(\Omega_N(\tau)\Big)^c\Big)<\exp(-N^C).$$
	 \end{coro}
	 \begin{proof}
	 	Suppose $S'$ is an open interval containing $\sigma_i(\kappa)$ and
	 	$S'\cap\sigma_j(\kappa)=\emptyset$ for all $j\ne i$. Suppose $f\in C^\infty(\mathbb{R})$  satisfies:
	 	\begin{enumerate}
	 		\item $f(x)=1$ if $x\in \sigma_i(\kappa)$,
	 		\item $f(x)=0$ if $x\not\in S'$,
	 		\item $0\le f(x)\le1$ if $x\in S'\backslash \sigma_i(\kappa)$.
	 	\end{enumerate}
	 	So $NR_i=N\int f(t)\rho(t)dt$, $\sum
	 	f(x_j)=\sharp\{j|x_j\in \sigma_i(\kappa)\}$ for all $x\in\Sigma_\kappa^{(N)}$ and
	 	\begin{align*}
	 		\P^{\mu_\kappa}\Big(\Big(\Omega_N^{(i)}(\tau)\Big)^c\Big)=	\P^{\mu_\kappa}(|\sharp\{j|x_j\in
	 		\sigma_i(\kappa)\}-NR_i|>N^\tau) =\P^{\mu_\kappa}(|\sum
	 		f(x_i)-N\int f(x)\rho(x)dx|>N^{\tau}).
	 	\end{align*}
	 	Applying Lemma \ref{lemma:fluctuation_for_mu_kappa} we complete the
	 	proof of the first inequality in the conclusion. This together with \eqref{eq:2252} yield the second inequality.
	 \end{proof}

	 \subsection{The measure $\mu_r$ on $\Sigma_\kappa^{(N)}$}\label{sec:mu_r}

	 The next measure $\mu_r$ was introduced by M. Shcherbina
	 \cite{Shcherbina,Shcherbina_fluctuation} and by Borot and Guionnet \cite{BG_multi_cut}. Let $\mu_r=\mu_r^{(N)}$ be a probability
	 measure on $\Sigma_\kappa^{(N)}$ with density
	 
	 \begin{align}
	 	\frac{1}{Z(\mu_r)}\exp\Big(-N\beta \mathcal H_r(x)\Big)
	 \end{align}
	 where
	 $$\mathcal H_r(x)=\frac{1}{2}\sum\limits_{i=1}^NV^{(r)}(x_i)-\frac{1}{2N}\sum\limits_{i\ne j}\ln|x_i-x_j|\sum\limits_{m=1}^q\mathds{1}_{\sigma_m(\kappa)}(x_i)\mathds{1}_{\sigma_m(\kappa)}(x_j)+\frac{N}{2}\Sigma^*$$
	 and
	 \begin{enumerate}
	 	\item $Z(\mu_r)$ is the normalization constant:
	 	$Z(\mu_r)=\int_{\Sigma_\kappa^{(N)}}\exp\Big(-N\beta\mathcal H_r(x)\Big)dx$,
	 	\item
	 	$V^{(r)}(x)=\sum\limits_{m=1}^q\mathds{1}_{\sigma_m(\kappa)}(x)\Big(V(x)-2\int_{\sigma\backslash\sigma_m}\rho(y)\ln|x-y|dy\Big)$,
	 	\item $\Sigma^*=\sum\limits_{i\ne
	 		j}\int_{\sigma_i}\int_{\sigma_j}\rho(x)\rho(y)\ln|x-y|dxdy$.
	 \end{enumerate}
	 
	 By the definition of $\mathcal H$ and $\mathcal H_r$ (see \eqref{eq:2202}), for
	 $x\in\Sigma_\kappa^{(N)}$,
	 \begin{align}\label{eq:230}
	 	\Delta \mathcal H(x):=\mathcal H_r(x)-\mathcal H(x)=\frac{1}{2N}\sum\limits_{i\ne j}\ln|x_i-x_j|\sum\limits_{m\ne m'}\mathds{1}_{\sigma_m(\kappa)}(x_i)\mathds{1}_{\sigma_{m'}(\kappa)}(x_j)-\sum\limits_{j=1}^N V^*(x_j)+\frac{N}{2}\Sigma^*
	 \end{align}
	 where
	 $V^*(x)=\sum\limits_{i=1}^q\mathds{1}_{\sigma_i(\kappa)}(x)\int_{\sigma\backslash\sigma_i}\ln|x-y|\rho(y)dy$.

	 \begin{proposition}\label{prop:small_for_mu_r_implies_small_for_mu_kappa}
	 	Suppose $A_N$ is a   Borel subset of
	 	$\Sigma_\kappa^{(N)}$. Suppose there exist $C_1>0$, $N_1>0$ such that
	 	$\P^{\mu_r}(A_N)\le\exp(-N^{C_1})$ when $N\ge N_1$. Then there exist
	 	$C_2>0$, $N_2>0$ 
	 	such that $$\P^{\mu_\kappa}(A_N)\le\exp(-N^{C_2})$$ when $N\ge N_2$.
	 \end{proposition}

	 \begin{proof}
	 	Without loss of generality, suppose $C_1<2$. For any $t>0$,
	 	\begin{multline*}
	 		\P^{\mu_\kappa}(A_N)=\P^{\mu_\kappa}(A_N\cap\{\Delta
	 		\mathcal H>t\})+\P^{\mu_\kappa}(A_N\cap\{\Delta \mathcal H\le t\})\\
	 		\le \P^{\mu_\kappa}(\Delta
	 		\mathcal H>t)+\frac{Z(\mu_r)}{Z(\mu_\kappa)}\exp(\beta Nt)\P^{\mu_r}(A_N)
	 	\end{multline*}
	 	By Jensen's inequality,
	 	$\ln\frac{Z(\mu_r)}{Z(\mu_\kappa)}=-\ln\E^{\mu_r}(\exp(\beta N\Delta
	 	\mathcal H))\le\E^{\mu_r}(-\beta N\Delta \mathcal H)\le\beta N|\E^{\mu_r}(\Delta
	 	\mathcal H)|$.

	 	By \cite{Shcherbina} there is $N_0>0$
	 	such that 
	 	\begin{align}\label{eq:253}
	 		|\E^{\mu_r}(\Delta \mathcal H)|\le  N^{-1+\frac{1}{3}C_1}
	 	\end{align}
	 	if $N>N_0$.  The inequality \eqref{eq:253} was proved in \cite{Shcherbina}.  See the fifth line on Page 785 of \cite{Shcherbina}. We remark that the main result of \cite{Shcherbina} is for $\beta\in\{1,2,4\}$ and polynomial $V$, but its proof of \eqref{eq:253} works for general $\beta>0$ and real analytic $V$.
	 	
	 	Setting $t=N^{-1+\frac{2}{3}C_1}$ we have for $N\ge N_0$:
	 	\begin{align}\label{eq:265}
	 		\P^{\mu_\kappa}(A_N)\le\P^{\mu_\kappa}(\Delta
	 		\mathcal H>N^{-1+\frac{2}{3}C_1})+\exp(\beta N^{\frac{1}{3}C_1}+\beta
	 		N^{\frac{2}{3}C_1}-N^{C_1}).
	 	\end{align}
	 	
	 	Now we use the method of Section 3 of \cite{Shcherbina} to estimate $\Delta\mathcal H$. By \eqref{eq:230},
	 	$$\Delta \mathcal H=\sum\limits_{m\ne m'}\Phi(m,m')$$
	 	where
	 	\begin{align*}
	 		\Phi(m,m')=&\frac{1}{2N}\sum\limits_{i\ne
	 			j}\ln|x_i-x_j|\mathds{1}_{\sigma_m(\kappa)}(x_i)\mathds{1}_{\sigma_{m'}(\kappa)}(x_j)\\
	 		-&\sum\limits_{j=1}^N\mathds{1}_{\sigma_m(\kappa)}(x_j)\int_{\sigma_{m'}}\ln|x_j-y|\rho(y)dy+\frac{N}{2}\int_{\sigma_m}\int_{\sigma_{m'}}\ln|x-y|\rho(x)\rho(y)dxdy
	 	\end{align*}
	 	
	 	Let $L=B_q-A_1+2$. Since
	 	$0<\kappa<\min(0.1,(A_2-B_1)/3,(A_3-B_2)/3,\ldots,(A_q-B_{q-1})/3)$,
	 	we have $|x-y|\in(\frac{1}{3}\min(A_2-B_1,\ldots,A_q-B_{q-1}),L-1)$
	 	for any $x\in\sigma_m(\kappa)$, $y\in\sigma_{m'}(\kappa)$
	 	with $m\ne m'$. So we can construct a function $g(x)$ such
	 	that
	 	\begin{enumerate}
	 		\item $g(x)$ depends only on $V$ and is independent of $\kappa$,
	 		\item $g(x)$ is smooth,
	 		\item $g(x)$ has a period $2L$,
	 		\item $g(x-y)=\ln|x-y|$ whenever $x\in\sigma_m(\kappa)$,
	 		$y\in\sigma_{m'}(\kappa)$ and $m\ne m'$.
	 	\end{enumerate}
	 	
	 	By Fourier transform,
	 	$$g(x)=\sum\limits_{k=-\infty}^{+\infty}c_k\exp(\frac{k\pi x}{L}{\rm i})$$
	 	where
	 	\begin{align}\label{eq:2204}
	 		c_k=\frac{1}{2L}\int_{-L}^Lg(x)\exp(-\frac{k\pi x}{L}{\rm i})dx.
	 	\end{align}
	 	(Here we should understand the sum $\sum_{k=-\infty}^{+\infty}$ as $\lim_{M\to\infty}\sum_{k=-M}^M$.)
	 	We have from the periodicity of $g$ that for any $p\in\{0,1,2,\ldots\}$ and $k\ne0$,
	 	\begin{align}\label{eq:246} 
	 		|c_k|=\Big|\frac{1}{2L}\Big(\frac{L}{k\pi{\rm i}}\Big)^p\int_{-L}^L\exp\big(-\frac{k\pi x}{L}{\rm i}\big)g^{(p)}(x)dx\Big|\le\Big|\frac{L}{k\pi}\Big|^p\|g^{(p)}\|_\infty.
	 	\end{align}

	 	Therefore when $m\ne m'$,
	 	\begin{multline}\label{eq:2203}
	 		\Phi(m,m')+\Phi(m',m)\\
	 		=\frac{1}{N}\sum\limits_{i\ne
	 			j}\ln|x_i-x_j|\mathds{1}_{\sigma_m(\kappa)}(x_i)\mathds{1}_{\sigma_{m'}(\kappa)}(x_j)-\sum\limits_{j=1}^N\mathds{1}_{\sigma_m(\kappa)}(x_j)\int_{\sigma_{m'}}\ln|x_j-y|\rho(y)dy\\
	 		-\sum\limits_{j=1}^N\mathds{1}_{\sigma_{m'}(\kappa)}(x_j)\int_{\sigma_{m}}\ln|x_j-y|\rho(y)dy+N\int_{\sigma_m}\int_{\sigma_{m'}}\ln|x-y|\rho(x)\rho(y)dxdy\\
	 		=\sum\limits_{k\in\mathbb{Z}}\frac{c_k}{N}\Big[\sum\limits_{i=1}^Ne^{\frac{{\rm i}k\pi x_i}{L}}\mathds{1}_{\sigma_m(\kappa)}(x_i)-N\int_{\sigma_m}e^{\frac{{\rm i}k\pi
	 				x}{L}}\rho(x)dx\Big]\Big[\sum\limits_{j=1}^Ne^{-\frac{{\rm i}k\pi x_j}{L}}\mathds{1}_{\sigma_{m'}(\kappa)}(x_j)-N\int_{\sigma_{m'}}e^{-\frac{{\rm i}k\pi
	 				y}{L}}\rho(y)dx\Big]\\
	 		=\frac{1}{N}\sum\limits_{k\in\mathbb{Z}}c_kI_k^m \overline{I_k^{m'}}
	 	\end{multline}
	 	where
	 	$I_k^m=\sum\limits_{i=1}^Ne^{{\rm i}k\pi x_i/L}\mathds{1}_{\sigma_m(\kappa)}(x_i)-N\int_{\sigma_m}e^{{\rm i}k\pi
	 		x/L}\rho(x)dx$. We remark that each sum in \eqref{eq:2203} converges absolutely and this fact can be seen from \eqref{eq:246} with $p=2$.
	 	
	 	Obviously $|I_k^m|\le2N$ and  $\Delta
	 	\mathcal H=\frac{1}{2N}\sum_{m\ne m'}\sum_{k\in\mathbb{Z}}c_kI_k^m
	 	\overline{I_k^{m'}}$.
	 	
	 	Set $w=\frac{C_1}{100}$, $p>10+\frac{400}{C_1}$ and $p\in\mathbb{N}$. Then
	 	$\frac{2}{3}C_1>3+w-wp$. It is easy to see that there is $N_0>0$  such that if $N>N_0$ and $m\ne m'$, then
	 	\begin{multline*}
	 		\Big|\sum\limits_{|k|>N^w}c_kI_k^{m}
	 		\overline{I_k^{m'}}\Big|\le\sum\limits_{|k|>N^w}|c_k|4N^2\le8N^2\sum\limits_{k>N^w}\Big(\frac{L}{k\pi}\Big)^p\|g^{(p)}\|_\infty\\
	 		\le\frac{10}{p-1}\Big(\frac{L}{\pi}\Big)^{p}\|g^{(p)}\|_\infty
	 		N^{2+w-wp}
	 		\le\frac{10}{p-1}\Big(\frac{L}{\pi}\Big)^{p}\|g^{(p)}\|_\infty
	 		N^{\frac{2}{3}C_1-1}.
	 	\end{multline*}
	 	Therefore if $N>N_0$, then
	 	$\Big|\sum\limits_{|k|>N^w}c_kI_k^{m}
	 	\overline{I_k^{m'}}\Big|<\dfrac{N^{\frac{2}{3}C_1}}{4q^2}$ for all
	 	$m$, $m'$ and
	 	\begin{align}\label{eq:266}
	 		&\P^{\mu_\kappa}\Big(\Delta\mathcal H>N^{-1+\frac{2}{3}C_1}\Big)=\P^{\mu_\kappa}\Big(\sum_{m\ne m'}\sum_{k\in\mathbb{Z}}c_kI_k^m
	 		\overline{I_k^{m'}}>2N^{\frac{2}{3}C_1}\Big)\le\sum_{m\ne m'}\P^{\mu_\kappa}\Big(\big|\sum_{k\in\mathbb{Z}}c_kI_k^m
	 		\overline{I_k^{m'}}\big|>\dfrac{N^{\frac{2}{3}C_1}}{q^2}\Big)\nonumber\\
	 		\le&\sum_{m\ne m'}\P^{\mu_\kappa}\Big(\big|\sum_{|k|\le
	 			N^w}c_kI_k^m
	 		\overline{I_k^{m'}}\big|>\dfrac{N^{\frac{2}{3}C_1}}{2q^2}\Big)
	 		\le\sum_{m\ne m'}\P^{\mu_\kappa}\Big(\sum_{|k|\le
	 			N^w}\big|I_k^m
	 		\overline{I_k^{m'}}\big|>\dfrac{N^{\frac{2}{3}C_1}}{2q^2}\|g\|_\infty^{-1}\Big)\quad\text{(by \eqref{eq:246})}\nonumber\\
	 		\le&\sum_{m\ne m'}\sum_{|k|\le
	 			N^w}\P^{\mu_\kappa}\Big(\big|I_k^m
	 		\overline{I_k^{m'}}\big|>\dfrac{N^{\frac{2}{3}C_1-w}}{6q^2}\|g\|_\infty^{-1}\Big)\nonumber\\
	 		\le&\sum_{m\ne m'}\sum_{|k|\le
	 			N^w}\Bigg[\P^{\mu_\kappa}\Big(\big|I_k^m
	 		\big|>\dfrac{1}{\sqrt{6\|g\|_\infty}q}N^{C_1/12}\Big)+\P^{\mu_\kappa}\Big(\big|
	 		\overline{I_k^{m'}}\big|>\dfrac{1}{\sqrt{6\|g\|_\infty}q}N^{C_1/12}\Big)\Bigg]
	 	\end{align}
	 	
	 	Suppose $h_{m,\kappa,k}(x)=\phi_{m,\kappa}(x)e^{{\rm i}k\pi x/L}$ where $\phi_{m,\kappa}\in C^\infty(\mathbb{R})$ such that $\phi_{m,\kappa}(x)=1$ for $x\in\sigma_m(\kappa)$ and $\phi_{m,\kappa}(x)=0$ for $x\in\sigma_{m'}(\kappa)$, $\forall m'\ne m$. Then we have:
	 	\begin{enumerate}
	 		\item the real and imaginary parts of  $h_{m,\kappa,k}$ are both $C^\infty$,
	 		\item  $h_{m,\kappa,k}(x)=e^{{\rm i}k\pi x/L}$, whenever
	 		$x\in\sigma_m(\kappa)$; $h_{m,\kappa,k}(y)=0$ whenever
	 		$x\in\sigma_{m'}(\kappa)$ and $m'\ne m$,
	 		\item there exists a constant $C_2>0$ such that for any $k\in\mathbb{Z}$, $$\sum\limits_{i=0}^6\|(\text{Re}h_{m,\kappa,k})^{(i)}\|_\infty+\sum\limits_{i=0}^6\|(\text{Im}h_{m,\kappa,k})^{(i)}\|_\infty\le C_2(1+k^6),$$
	 		\item both the real part and the imaginary part of $h_{m,\kappa,k}(x)$ can be analytically extended to $O:=\{z\in\mathbb C|\text{dist}(z,\sigma)<\kappa/2\}$ and they are bounded by $1$ on $O$.
	 	\end{enumerate}
	 	
	 	Using Lemma \ref{lemma:fluctuation_for_mu_kappa} for the real and
	 	imaginary parts of $h_{m,\kappa,k}/N^{6w}$, we have that there exist constants $N_0>0$ and $C_3>0$
	 	such  that if $N>N_0$,
	 	then for $|k|\le N^w$
	 	\begin{align}\label{eq:267}
	 		\P^{\mu_\kappa}\Big(\big|I_k^m
	 		\big|>\dfrac{1}{\sqrt{6\|g\|_\infty}q}N^{C_1/12}\Big)=&\P^{\mu_\kappa}\Big(\big|\sum
	 		h_{m,\kappa,k}(x_i)-N\int
	 		h_{m,\kappa,k}(x)\rho(x)dx\big|>\dfrac{1}{\sqrt{6\|g\|_\infty}q}N^{C_1/12}\Big)\nonumber\\
	 		=&\P^{\mu_\kappa}\Big(\big|\sum
	 		\frac{h_{m,\kappa,k}(x_i)}{N^{6w}}-N\int
	 		\frac{h_{m,\kappa,k}(x)}{N^{6w}}\rho(x)dx\big|>\dfrac{1}{\sqrt{6\|g\|_\infty}q}N^{\frac{C_1}{12}-6w}\Big)\nonumber\\
	 		\le&\exp(-N^{C_3}).
	 	\end{align}
	 	(To use Lemma \ref{lemma:fluctuation_for_mu_kappa} we need to consider $\frac{h_{m,\kappa,k}(x)}{N^{6w}}$ instead of $h_{m,\kappa,k}(x)$ since Lemma \ref{lemma:fluctuation_for_mu_kappa} requires the test function and its first six derivatives are all bounded by a constant.) By \eqref{eq:265}, \eqref{eq:266} and \eqref{eq:267}, if $N>N_0$, then for some constant $C_4>0$,
	 	\begin{align*}
	 		\P^{\mu_\kappa}(A_N)\le&\P^{\mu_\kappa}(\Delta
	 		\mathcal H>N^{-1+\frac{2}{3}C_1})+\exp(\beta N^{\frac{1}{3}C_1}+\beta
	 		N^{\frac{2}{3}C_1}-N^{C_1})\\
	 		\le&q^2(2N^w+1)\cdot2\exp(-N^{C_3})+\exp(\beta N^{\frac{1}{3}C_1}+\beta
	 		N^{\frac{2}{3}C_1}-N^{C_1})\\
	 		\le&\exp(-N^{C_4}).
	 	\end{align*}
	 \end{proof}
	 
	 \subsection{The measures $\nu^{(i,M,\xi)}$  on $\Sigma_\kappa^{(M)}(i)$}\label{sec:the measure nu(i,N,cN)} Recall that $R_i=\int_{A_i}^{B_i}\rho(t)dt$ for $1\le i\le q$. Set
	 \begin{align}\label{eq:2205}
	 	\Sigma_\kappa^{(M)}(i):=\{(x_1,\ldots,x_M)\in\sigma_i(\kappa)^M |x_1\le\cdots\le x_M\}.
	 \end{align}
	 
	 \begin{definition}
	 	For any $i\in\{1,\ldots,q\}$, $M\in\mathbb N$ and $\xi>0$,	define $\nu^{(i,M,\xi)}$ to be a probability measure on
	 	$\Sigma_\kappa^{(M)}(i)$ with density
	 	\begin{align}\label{eq:237} 
	 		\dfrac{1}{Z(\nu^{(i,M,\xi)})}\exp\Big(-\frac{M\beta}{2}\sum\limits_{j=1}^M\xi\cdot\frac{1}{R_i}[V(x_j)-2\int_{\sigma\backslash\sigma_{i}}\ln|x_j-y|\rho(y)dy]\Big)\prod\limits_{1\le
	 			u<v\le M}|x_u-x_v|^\beta
	 	\end{align}
	 	where $Z(\nu^{(i,M,\xi)})$ is the normalization constant.
	 \end{definition}
	 Notice that $\nu^{(i,M,\xi)}$ has the same form as $\mu_1$ defined in Definition \ref{def:measures} with $N$ replaced by $M$, $c_N$ replaced by $\xi$ and $V_0$ replaced by
	 \begin{align}\label{eq:2212}
	 	\frac{1}{R_i}[V(x)-2\int_{\sigma\backslash\sigma_{i}}\ln|x-y|\rho(y)dy].
	 \end{align}
	 
	 The factor $\frac{1}{R_i}$ in \eqref{eq:2212} is important because it makes the potential function  satisfy the Euler-Lagrange condition for the $probability$ measure $\frac{1}{R_i}\rho(x)\mathds{1}_{\sigma_i}(x)dx$:
	 \begin{multline}\label{eq:2210}
	 	\Bigg[\frac{1}{R_i}[V(x)-2\int_{\sigma\backslash\sigma_{i}}\ln|x-y|\rho(y)dy]\Bigg]-2\int_{\mathbb R}\frac{1}{R_i}\rho(y)\mathds{1}_{\sigma_i}(y)\ln|x-y|dy\\
	 	\begin{cases}=\min_{t\in\mathbb R}\frac{1}{R_i}\Big[V(t)-2\int_\sigma\rho(y)\ln|t-y|dy\Big]\quad\forall x\in \sigma_i\\
	 		>\min_{t\in\mathbb R}\frac{1}{R_i}\Big[V(t)-2\int_\sigma\rho(y)\ln|t-y|dy\Big]\quad\forall x\in \sigma_i(\kappa)\backslash\sigma_i
	 	\end{cases}
	 \end{multline}
	 where we used the last assumption of Hypothesis \ref{Hypothesis for initial model} and the fact that the left hand side of \eqref{eq:2210} equals $\frac{1}{R_i}\Big[V(x)-2\int_\sigma\rho(y)\ln|x-y|dy\Big]$. Notice that \eqref{eq:2212} can be smoothly extended to $\mathbb R$ such that the identity in \eqref{eq:2210} holds for $x\in \sigma_i$ and the indequality in \eqref{eq:2210} holds for $x\in \mathbb R\backslash\sigma_i$. So by  the fourth conclution of Lemma \ref{lemma:minimizer} and Lemma \ref{coro:convergence_for_restriction}, if $\lim_{M\to\infty}\xi=1$, then the empirical measure of $\nu^{(i,M,\xi)}$ converges almost surely and in expectation to
	 \begin{align}\label{eq:2253}
	 	\frac{1}{R_i}\rho(x)\mathds{1}_{\sigma_i}(x)dx
	 \end{align}
	 as $M\to\infty$.  and therefore
	 $$\lim\limits_{M\to\infty}\frac{1}{M}\E^{\nu^{(i,M,\xi)}}\Big[\sum\limits_{j=1}^Mf(x_j)\Big]=\int_{\mathbb{R}}\frac{1}{R_i}\rho(x)\mathds{1}_{\sigma_i}(x)f(x)dx$$
	 as long as $f:\mathbb R\to\mathbb R$ is a bounded continuous function.
	 
	 \begin{definition}
	 	Define the classical position $\theta(i,M,k)$ of the $k$-th particle under $\nu^{(i,M,\xi)}$ by
	 	\begin{align}\label{definition of classical position for nu_s(i,N,c_N)}
	 		\theta(i,M,k)=\inf\Big\{x\Big|	\int_{A_i}^x\frac{1}{R_i}\rho(x)dx=\frac{k}{M}\Big\}\quad(1\le k\le M)
	 	\end{align}
	 \end{definition}
	 
	 \begin{lemma}\label{lemma:important}
	 	Suppose $\nu^{(i,M,\xi)}$ is defined as \eqref{eq:237}.	For any $\epsilon>0$, there exist constants $M_0>0$ and
	 	$C\in(0,0.01)$  such that
	 	if $1\le i\le q$, $\epsilon_0<C/10$, $M>M_0$ and $|\xi-1|< M^{-1+\epsilon_0}$, then
	 	\begin{align*}
	 		\P^{\nu^{(i,M,\xi)}}(\exists k\in[1,M]\text{ such that }|x_k-\theta(i,M,k)|>M^{-\frac{2}{3}+\epsilon}\cdot\big(\min(k,M+1-k)\big)^{-1/3})<e^{-M^C}
	 	\end{align*}
	 \end{lemma}
	 
	 \begin{proof}
	 	This lemma is  a direct application of Corollary \ref{thm:rigidity_for_mu_1}. To see this,	let
	 	\begin{align*}
	 		V_i(x)=\frac{1}{R_i}[V(x)-2\int_{\sigma\backslash\sigma_{i}}\ln|x-y|\rho(y)dy].
	 	\end{align*}
	 	By the property of convolution we have:
	 	\begin{itemize}
	 		\item  $V_i\in C^0(\mathbb{R})$;
	 		\item  $x\mapsto\int_{A_j}^{B_j}\rho(y)\ln|x-y|dy$ can be analytically extended either to $\mathbb{C}\backslash(-\infty,B_j]$ or to $\mathbb{C}\backslash[A_j,+\infty)$. 
	 	\end{itemize}
	 	So by the definition of $\kappa$, $V_i$ can be analytically extended to $\{z\in\mathbb{C}|\text{dist}(z,[A_i,B_i])<30\kappa\}$. According to Hypothesis \ref{Hypothesis for initial model}, 
	 	\begin{align*}
	 		&V_i(x)-2\int_{A_i}^{B_i}\frac{\rho(y)}{R_i}\ln|x-y|dy\\&\begin{cases}
	 			=\min\limits_{t\in[A_i-2\kappa,B_i+2\kappa]}\big[V_i(t)-2\int_{A_i}^{B_i}\frac{\rho(y)}{R_i}\ln|t-y|dy\big]\quad\text{if }x\in[A_i,B_i],\\>\min\limits_{t\in[A_i-2\kappa,B_i+2\kappa]}\big[V_i(t)-2\int_{A_i}^{B_i}\frac{\rho(y)}{R_i}\ln|t-y|dy\big]\quad\text{if }x\in[A_i-2\kappa,B_i+2\kappa]\backslash[A_i,B_i].
	 		\end{cases}
	 	\end{align*}
	 	Let:
	 	\begin{enumerate}
	 		\item $p(x)=C_1x^2$ with the coefficient $C_1>0$ large enough such that for every $x\in\mathbb{R}$: $$p(x)-2\int_{A_i}^{B_i}\frac{\rho(y)}{R_i}\ln|x-y|dy>\min\limits_{t\in[A_i-2\kappa,B_i+2\kappa]}\big[V_i(t)-2\int_{A_i}^{B_i}\frac{\rho(y)}{R_i}\ln|t-y|dy\big];$$ 
	 		\item $\phi(x):\mathbb{R}\to[0,1]$ be a smooth function such that $\phi(x)=1$ if $x\in [A_i-\kappa,B_i+\kappa]$ and $\phi(x)=0$ if $x\not\in [A_i-2\kappa,B_i+2\kappa]$;
	 		\item $U(x)=\phi(x)V_i(x)+(1-\phi(x))p(x)$. 
	 	\end{enumerate} 
	 	According to the construction, $U(x)$ satisfies the following conditions.
	 	\begin{itemize}
	 		\item $U(x)\in C^\infty(\mathbb{R})$ and $\liminf\limits_{|x|\to\infty}\frac{U(x)}{\ln|x|}>2$;
	 		\item $U(x)=V_i(x)$ for $x\in[A_i-\kappa,B_i+\kappa]$;
	 		\item $\inf\limits_{x\in\mathbb{R}}U''(x)>-\infty$;
	 		\item \begin{align*}
	 			U(x)-2\int_{A_i}^{B_i}\frac{\rho(y)}{R_i}\ln|x-y|dy\begin{cases}
	 				=\min\limits_{t\in\mathbb{R}}\big[U(t)-2\int_{A_i}^{B_i}\frac{\rho(y)}{R_i}\ln|t-y|dy\big]\quad\text{if }x\in[A_i,B_i],\\>\min\limits_{t\in\mathbb{R}}\big[U(t)-2\int_{A_i}^{B_i}\frac{\rho(y)}{R_i}\ln|t-y|dy\big]\quad\text{if }x\in\mathbb{R}\backslash[A_i,B_i].
	 			\end{cases}
	 		\end{align*}
	 	\end{itemize}
	 	Use Corollary \ref{thm:rigidity_for_mu_1} with
	 	\begin{itemize}
	 		\item $V_0(x)=U(x)$, $[a,b]=\sigma_i(\kappa)=[A_i-\frac{\kappa}{2},B_i+\frac{\kappa}{2}]$, $[c,d]=[A_i,B_i]$, $c_N=\xi$, $N=M$,
	 		\item $D=\{z\in\mathbb{C}|\text{dist}(z,[A_i,B_i])<\kappa\}$,
	 		\item $r_0(x)=\frac{{\rm i}r(x)}{R_i}\cdot\prod\limits_{j\ne i}[\sqrt{x-A_j}\sqrt{x-B_j}]$ (recall \eqref{eq:239}) 
	 		\item $\rho_0(x)=r_0(x)\sqrt{(x-A_i)(B_i-x)}\mathds{1}_{[A_i,B_i]}(x)=\frac{\rho(x)}{R_i}\mathds{1}_{[A_i,B_i]}(x),$
	 	\end{itemize}
	 	then we complete the proof. 
	 \end{proof}

	 \section{Proof of the main theorem}\label{sec:proof of main thm}
	 
	 Recall that $\eta_k=\eta_k^{(N)}$ is defined in \eqref{eq:268} and $\theta(i,M,
	 k)$ is defined in \eqref{definition of classical position for nu_s(i,N,c_N)}.
	 \begin{lemma}\label{lemma:lemma_on_k_i}
	 	Let $\alpha\in(0,\min_i\frac{R_i}{2})$ be an arbitrarily small constant. Suppose $c_1\in(0,1)$. There exists a constant $N_0>0$ such that if 
	 	\begin{itemize}
	 		\item[(i)] $N>N_0$,
	 		\item[(ii)] $k_1,\ldots,k_q$ are positive integers with $k_1+\cdots+k_q=N$,
	 		\item[(iii)] $|k_j-NR_j|\le N^{c_1}$ for $1\le j\le q$,
	 	\end{itemize}
	 	then for any $i$ in $\{1,\ldots,q\}$ we have
	 	\begin{enumerate}
	 		\item $[(R_1+\cdots+R_{i-1}+\alpha)N,(R_1+\cdots+R_i-\alpha)N]\subset[k_1+\cdots+k_{i-1}+\frac{\alpha}{2R_i}k_i,k_1+\cdots+k_i-\frac{\alpha}{2R_i}k_i]$
	 		\item $[k_1+\cdots+k_{i-1}+\frac{\alpha}{2R_i}k_i,k_1+\cdots+k_i-\frac{\alpha}{2R_i}k_i]\subset[(R_1+\cdots+R_{i-1}+\frac{\alpha}{3})N,(R_1+\cdots+R_i-\frac{\alpha}{3})N]$
	 		\item $|\frac{NR_s}{k_s}-1|\le k_i^{-1+2c_1}$, for all $s\in\{1,\ldots,q\}$
	 		\item$
	 		|\eta_k-\theta(i,k_i,\bar
	 		k)|\le CN^{-1+c_1},\quad\forall k\in[k_1+\cdots+k_{i-1}+\frac{\alpha}{2R_i} k_i,k_1+\cdots+k_i-\frac{\alpha}{2R_i} k_i]$
	 		\item $
	 		|\eta_k-\theta(1,k_1,
	 		k)|\le C N^{-\frac{5}{3}+c_1}\cdot k^{2/3},\quad\forall k\in[1,\alpha N]$
	 		\item $
	 		|\eta_k-\theta(q,k_q,
	 		k-(k_1+\cdots+k_{q-1}))|\le C N^{-\frac{5}{3}+c_1}\cdot(N+1-k)^{2/3},\quad\forall k\in[(1-\alpha) N,N]$
	 	\end{enumerate}
	 	where $\bar k:=k-(k_1+\cdots+k_{i-1})$ and $C>0$ is a constant.
	 \end{lemma}
	 \begin{proof}
	 	The first three conclusions are trivial. According to the second conclusion, if 
	 	\begin{itemize}
	 		\item $N>N_0$ and (ii), (iii) are satisfied
	 		\item $k\in[k_1+\cdots+k_{i-1}+\frac{\alpha}{2R_i} k_i,k_1+\cdots+k_i-\frac{\alpha}{2R_i} k_i]$
	 	\end{itemize}
	 	then both   $\eta_k$ and
	 	$\theta(i,k_i,\bar k)$ are in
	 	$[A_i+\epsilon_1,B_i-\epsilon_1]$ for some small constant $\epsilon_1>0$. In this case, let $\mathcal{M}:=\min\limits_{x\in[A_i+\epsilon_1,B_i-\epsilon_1]}\rho(x)>0$ and we have
	 	\begin{align*}
	 		&|\eta_k-\theta(i,k_i,\bar
	 		k)|\frac{\mathcal{M}}{R_i}\le\Big|\int_{\eta_k}^{\theta(i,k_i,\bar
	 			k)}\frac{1}{R_i}\rho(x)dx\Big|
	 		=\Big|\int_{A_i}^{\theta(i,k_i,\bar k)}\frac{1}{R_i}\rho(x)dx-\int_{A_i}^{\eta_k}\frac{1}{R_i}\rho(x)dx\Big|\\
	 		=&\Big|\dfrac{\bar
	 			k}{k_i}-\frac{1}{R_i}(\frac{k}{N}-(R_1+\cdots+R_{i-1}))\Big|\\
	 		=&\Big|\dfrac{\bar
	 			k(NR_i-k_i)}{NR_ik_i}+\frac{1}{NR_i}\big[(NR_1-k_1)+\cdots+(NR_{i-1}-k_{i-1})\big]\Big|\\
	 		\le&\Big|\dfrac{(NR_i-k_i)}{NR_i}\Big|+\frac{1}{NR_i}\big(|NR_1-k_1|+\cdots+|NR_{i-1}-k_{i-1}|\big)
	 		\le\frac{i}{R_i}N^{-1+c_1}\le\frac{q}{R_i}N^{-1+c_1}
	 	\end{align*}
	 	so $|\eta_k-\theta(i,k_i,\bar
	 	k)|\le \frac{q}{\mathcal M}N^{-1+c_1}$ and we proved the fourth conclusion.

	 	By definition we have
	 	$
	 	\int_{A_1}^{\eta_k}\rho(t)dt=\frac{k}{N}$ and $\int_{A_1}^{\theta(1,k_1,k)}\rho(t)/R_1dt=\frac{k}{k_1}$.
	 	Since $\alpha<R_1/2$, there exist constants $M_1>0$, $M_2>0$ and $N_0>0$ such that if $N>N_0$, $k\in[1,\alpha N]$ and (ii), (iii) are satisfied, 
	 	then 
	 	\begin{align*}
	 		M_2\sqrt{t-A_1}\ge	\rho(t)\ge M_1\sqrt{t-A_1},\quad\quad\forall t\in[A_1,\max(\eta_k,\theta(1,k_1,k))]
	 	\end{align*}
	 	and thus
	 	\begin{itemize}
	 		\item 	\begin{align}\label{eq:277}
	 			M_1\Big|\int_{\eta_k}^{\theta(1,k_1,k)}\sqrt{t-A_1}dt\Big|\le\Big|\int_{\eta_k}^{\theta(1,k_1,k)}\rho(t)dt\Big|=\frac{k}{N}\Big|\frac{NR_1}{k_1}-1\Big|\le \frac{k N^{c_1}}{Nk_1}
	 		\end{align}
	 		\item 
	 		\begin{align}\label{eq5}
	 			\frac{k}{N}=\int_{A_1}^{\eta_k}\rho(t)dt\le M_2\int_{A_1}^{\eta_k}\sqrt{t-A_1}dt=\frac{2M_2}{3}(\eta_k-A_1)^{3/2}
	 		\end{align}
	 		\item 
	 		\begin{align}\label{eq6}
	 			\frac{k}{k_1}=\int_{A_1}^{\theta(1,k_1,k)}\rho(t)/R_1dt\le \frac{M_2}{R_1}\int_{A_1}^{\theta(1,k_1,k)}\sqrt{t-A_1}dt=\frac{2M_2}{3R_1}(\theta(1,k_1,k)-A_1)^{3/2}
	 		\end{align}
	 	\end{itemize}
	 	According to \eqref{eq5} and \eqref{eq6}, if $N>N_0$ and $k\in[1,\alpha N]$, then 
	 	$$|\eta_k-A_1|\ge(\frac{3}{2M_2}\frac{k}{N})^{2/3},$$ 
	 	$$|\theta(1,k_1,k)-A_1|\ge (\frac{3R_1}{2M_2}\cdot\frac{k}{k_1})^{2/3}\ge (\frac{3}{4M_2}\cdot\frac{k}{N})^{2/3}$$
	 	and therefore
	 	\begin{align}\label{eq7}
	 		\Big|\int_{\eta_k}^{\theta(1,k_1,k)}\sqrt{t-A_1}dt\Big|\ge |\theta(1,k_1,k)-\eta_k|\sqrt{\min(\eta_k,\theta(1,k_1,k))-A_1}\ge|\theta(1,k_1,k)-\eta_k|\cdot(\frac{k}{N})^{1/3}\cdot C_1
	 	\end{align}
	 	where $C_1>0$ is a constant. By \eqref{eq:277} and \eqref{eq7}, if $N>N_0$ and $k\in[1,\alpha N]$, then 
	 	\begin{multline*}
	 		|\theta(1,k_1,k)-\eta_k|\le\frac{1}{C_1}(\frac{k}{N})^{-1/3} \Big|\int_{\eta_k}^{\theta(1,k_1,k)}\sqrt{t-A_1}dt\Big| 
	 		\\
	 		\le  \frac{1}{C_1}(\frac{k}{N})^{-1/3}\cdot \frac{1}{M_1}\frac{k N^{c_1}}{Nk_1}\le\frac{1}{C_1M_1}(\frac{k}{N})^{-1/3}\cdot \frac{k N^{c_1}}{NR_1N/2}.
	 	\end{multline*}
	 	So the fifth conclusion is proved. The sixth conclusion can be proved in the same way as the fifth one, or by changing the potential $V(x)\mapsto V(-x)$.
	 \end{proof}

	 Now we are ready to prove Theorem \ref{thm:main_thm}. Suppose $\alpha>0$ and $\epsilon>0$ are as in Theorem \ref{thm:main_thm}. Without loss of generality, suppose $\epsilon\in(0,0.01)$.
	 
	 According to Lemma \ref{lemma:important}
	 there exist constants $M_0>0$ and
	 $c_0\in(0,0.01)$  such that
	 if $1\le i\le q$, $\epsilon_0<c_0/10$, $M>M_0$ and $|\xi-1|< M^{-1+\epsilon_0}$, then
	 \begin{align}\label{eq:272}
	 	\P^{\nu^{(i,M,\xi)}}(\exists k\in[\frac{\alpha M}{2\max\limits_iR_i} ,(1-\frac{\alpha }{2\max\limits_iR_i})M]\text{ s.t. }|x_k-\theta(i,M,k)|>M^{-1+\frac{\epsilon}{2}})
	 	<\exp(-M^{c_0})
	 \end{align}
	 and
	 \begin{multline}\label{eq:278}
	 	\P^{\nu^{(i,M,\xi)}}\Bigg(\exists k\in[1,\frac{\alpha M}{0.9\min\limits_iR_i}]\cup[(1-\frac{\alpha }{0.9\min\limits_iR_i})M,M]\\
	 	\text{ such that }|x_k-\theta(i,M,k)|>M^{-\frac{2}{3}+\frac{\epsilon}{2}}\cdot(\min(k,M+1-k))^{-\frac{1}{3}}\Bigg)
	 	<\exp(-M^{c_0}) 
	 \end{multline}
	 Set  $\epsilon_*=\min(\frac{1}{100}c_0,\frac{1}{100}\epsilon)<0.0001$. Using the notation $\Omega_N(\cdot)$ defined in  Corollary \ref{lemma:initial estimate 3}:
	 $$\Omega_N(\epsilon_*)=\Big\{x\in\Sigma_\kappa^{(N)}\Big||\sharp\{j|x_j\in\sigma_i(\kappa)\}-NR_i|\le
	 N^{\epsilon_*},\forall 1\le i\le q\Big\}.$$
	 \begin{lemma}\label{lemma:measures}
	 	Suppose $E_N$ is a Borel subset of $\Sigma^{(N)}$. If there exist $\tau_1>0$ and $N_1>0$ such that $\P^{\mu_r}(E_N\cap\Omega_N(\epsilon_*))\le\exp(-N^{\tau_1})$ when $N>N_1$, then there exist  $\tau_2>0$ and $N_2>0$ such that $\P^\mu(E_N)\le\exp(-N^{\tau_2})$ for $N>N_2$.
	 \end{lemma}
	 \begin{proof}
	 	By Corollary \ref{lemma:initial estimate 3}, there are constants $C_1>0$ and
	 	$N_0>0$  such
	 	that if $N>N_0$, then
	 	\begin{align}
	 		\P^{\mu_\kappa}(\Omega_N(\epsilon_*)^c)\le\exp(-N^{C_1}).
	 	\end{align}
	 	
	 	According to Proposition
	 	\ref{prop:small_for_mu_r_implies_small_for_mu_kappa}, there exist constants
	 	$C_2>0$, $N_0>0$  such that if $N>N_0$, then
	 	$\P^{\mu_\kappa}(E_N\cap\Omega_N(\epsilon_*))\le\exp(-N^{C_2})$ and 
	 	\begin{align} \label{eq:273}
	 		\P^{\mu_\kappa}(E_N\cap\Sigma_\kappa^{(N)})\le\exp(-N^{C_2})+\exp(-N^{C_1}).
	 	\end{align}
	 	Notice that
	 	\begin{align} \label{eq:274}
	 		\P^\mu(E_N)\le\P^\mu(E_N\cap\Sigma_\kappa^{(N)})+\P^\mu(\Sigma^{(N)}\backslash\Sigma_\kappa^{(N)}).
	 	\end{align}
	 	According to Lemma \ref{lemma:minimizer}, there exist constants
	 	$C_3>0$, $N_0>0$  such that if $N>N_0$,
	 	then
	 	\begin{align} \label{eq:275}
	 		\P^\mu(\Sigma^{(N)}\backslash\Sigma_\kappa^{(N)})<\exp(-C_3N)
	 	\end{align}
	 	On the other hand, by Lemma \ref{lemma:trivial},
	 	\begin{align} \label{eq:276}
	 		\P^\mu(E_N\cap\Sigma_\kappa^{(N)})=\P^\mu(\Sigma_\kappa^{(N)})\P^{\mu_\kappa}(E_N\cap\Sigma_\kappa^{(N)})\le\P^{\mu_\kappa}(E_N\cap\Sigma_\kappa^{(N)}).
	 	\end{align}
	 	
	 	According to \eqref{eq:273}, \eqref{eq:274}, \eqref{eq:275},
	 	\eqref{eq:276}, the lemma is proved.
	 	
	 \end{proof} 
	 Suppose $A_N$ is a Borel subset of $\Sigma_\kappa^{(N)}$, then
	 \begin{align}\label{eq:271}
	 	\P^{\mu_r}(A_N\cap\Omega_N(\epsilon_*))=\dfrac{\int_{\Sigma_\kappa^{(N)}}\exp(-N\beta\mathcal H_r(x))\mathds{1}_{A_N\cap\Omega_N(\epsilon_*)}(x)dx}{\int_{\Sigma_\kappa^{(N)}}\exp(-N\beta\mathcal H_r(x))dx}=\dfrac{\sum\limits_{k_1+\cdots+k_q=N}Q_{(k_1,\ldots,k_q)}\cdot\Phi^{(A_N)}_{k_1,\ldots,k_q}}{\sum\limits_{k_1+\cdots+k_q=N}Q_{(k_1,\ldots,k_q)}}
	 \end{align}
	 where
	 $$Q_{(k_1,\ldots,k_q)}=\int_{\Sigma_\kappa^{(N)}}\Upsilon(k_1,\ldots,k_q)\exp\Big(-N\beta\mathcal H_r(x)\Big)dx,$$
	 \begin{align}\label{eq:2211}
	 	\Phi^{(A_N)}_{k_1,\ldots,k_q}=\dfrac{\int_{\Sigma_\kappa^{(N)}}\Upsilon(k_1,\ldots,k_q)\exp\Big(-N\beta\mathcal H_r(x)\Big)\mathds{1}_{A_N\cap\Omega_N(\epsilon_*)}(x)dx}{Q_{(k_1,\ldots,k_q)}}
	 \end{align}
	 and
	 \begin{align}\label{eq:2254}
	 	\Upsilon(k_1,\ldots,k_q)=\prod_{l=1}^q\Bigg[\prod_{j=k_1+\cdots+k_{l-1}+1}^{k_1+\cdots+k_l}\mathds1_{\sigma_l(\kappa)(x_j)}\Bigg].
	 \end{align}
	 
	 \begin{remark}
	 	We learnt the idea of considering the $q$-tuple $(k_1,\ldots,k_q)$ from Section 3 of \cite{Shcherbina}. 
	 \end{remark}
	 
	 From the definition of $\mathcal H_r$, we see that in the integral of
	 $Q_{(k_1,\ldots,k_q)}$ the particles in different cuts don't have
	 intersection. Thus the integral can be written as a product of
	 integrals over domains with lower dimensions, i.e.,
	 \begin{align}\label{eq:270}
	 	Q_{(k_1,\ldots,k_q)}=Q_{(k_1,\ldots,k_q)}^{(1)}\cdots
	 	Q_{(k_1,\ldots,k_q)}^{(q)}\cdot\exp(-\frac{\beta}{2}N^2\Sigma^*)
	 \end{align}
	 where
	 \begin{multline*}
	 	Q_{(k_1,\ldots,k_q)}^{(i)}\\=\int_{\Sigma_\kappa^{(k_i)}(i)}\exp\Big(-\frac{N\beta}{2}\sum\limits_{j=1}^{k_i}\big(V(x_j)-2\int_{\sigma\backslash\sigma_i}\ln|x_j-y|\rho(y)dy\big)\Big)\prod\limits_{u<v}|x_u-x_v|^\beta
	 	dx_1\cdots dx_{k_i}
	 \end{multline*}
	 and
	 $\Sigma_\kappa^{(n)}(i)=\{(x_1,\ldots,x_n)|x_1\le\cdots\le x_n,x_j\in\sigma_i(\kappa),\forall1\le
	 j\le n\}$ as defined in \eqref{eq:2205}. It is easy to see that the value of $Q_{(k_1,\ldots,k_q)}^{(i)}$ is independent of $k_1,\ldots,k_{i-1},k_{i+1},\ldots,k_q$.

	 \subsection{Rigidity in the bulk}
	 Fix $1\le i_0\le q$. We make the convention that
	 \begin{itemize}
	 	\item for  $(k_1,\ldots,k_q)\in\mathbb N^q$ satisfying $k_1+\cdots+k_q=N$, let $\bar k=k-(k_1+\cdots+k_{i_0-1})$ and $\textbf{k}_{i_0}=\{k_1+\cdots+k_{i_0-1}+1,\ldots,k_1+\cdots+k_{i_0}\}$;
	 	\item  for $(x_1,\ldots,x_N)\in\mathbb R^N$, let $\bar x=(x_{k_1+\cdots+k_{i_0-1}+1},\ldots,x_{k_1+\cdots+k_{i_0}})$.
	 \end{itemize}

	 Set 
	 $$A_N'(i_0)=\{x\in\Sigma_\kappa^{(N)}|\exists
	 k\in[(R_1+\cdots+R_{i_0-1}+\alpha)N,(R_1+\cdots+R_{i_0}-\alpha)N]\text{
	 	s.t. }|x_k-\eta_k|>N^{-1+\epsilon}\}$$
	 According to  Lemma \ref{lemma:measures}, to prove the first conclusion of Theorem \ref{thm:main_thm}, we only need to show that $\P^{\mu_r}(A_N'(i_0)\cap\Omega_N(\epsilon_*))$ is exponentially small. Writting $\P^{\mu_r}(A_N'(i_0)\cap\Omega_N(\epsilon_*))$ in the form of \eqref{eq:271}, to control the factor $\Phi^{(A_N'(i_0))}_{k_1,\ldots,k_q}$ in each term in the numerator of the right hand side of \eqref{eq:271}, we consider the identity: 
	 \begin{align}\label{eq:269}
	 	\Upsilon(k_1,\ldots,k_q)\mathds{1}_{A_N'(i_0)\cap\Omega_N(\epsilon_*)}(x)=1
	 \end{align}
	 where $\Upsilon(k_1,\ldots,k_q)$ is defined in \eqref{eq:2254}. If \eqref{eq:269} does not hold, then the integrand in the numerator of the definition of $\Phi^{(A_N'(i_0))}_{k_1,\ldots,k_q}$ is zero. See \eqref{eq:2211}. If \eqref{eq:269} does  hold, then there are $k_i$ particles in the $i-$th cut and
	 \begin{enumerate}
	 	\item for $1\le j\le q$ we have $|k_j-NR_j|\le N^{\epsilon_*}$
	 	\item there exists
	 	$k\in[(R_1+\cdots+R_{i_0-1}+\alpha)N,(R_1+\cdots+R_{i_0}-\alpha)N]$
	 	such that $|x_k-\eta_k|>N^{-1+\epsilon}$.
	 \end{enumerate}
	 By Lemma \ref{lemma:lemma_on_k_i}, if $N>N_0$ and \eqref{eq:269} is true, then there exist $k\in[k_1+\cdots+k_{i_0-1}+\frac{\alpha}{2R_{i_0}}k_{i_0},k_1+\cdots+k_{i_0}-\frac{\alpha}{2R_{i_0}}k_{i_0}]$ and a constant $C_1>0$ such that
	 $$|x_k-\theta(i_0,k_{i_0},\bar
	 k)|\ge|x_k-\eta_k|-|\eta_k-\theta(i_0,k_{i_0},\bar
	 k)|\ge
	 N^{-1+\epsilon}-C_1N^{-1+\epsilon_*}>\frac{1}{2}N^{-1+\epsilon}>k_{i_0}^{-1+{\frac{1}{2}\epsilon}}$$
	 therefore
	 $$\bar{x}\in \tilde{\Omega}_N(i_0):=\{y\in\Sigma_\kappa^{(k_{i_0})}(i_0)|\exists
	 k\in[\frac{\alpha}{2R_{i_0}}k_{i_0},k_{i_0}-\frac{\alpha}{2R_{i_0}}k_{i_0}]\text{
	 	s.t. } |y_k-\theta(i_0,k_{i_0},
	 k)|>k_{i_0}^{-1+\frac{1}{2}\epsilon}\}$$
	 and (according to the definition of $\Phi^{(A_N)}_{k_1,\ldots,k_q}$ and \eqref{eq:270})
	 \begin{align*}
	 	&\Phi^{(A_N'(i_0))}_{k_1,\ldots,k_q}\\
	 	\le&\frac{\int_{\Sigma_\kappa^{(k_{i_0})}(i_0)}\exp\Big[-\frac{N\beta}{2}\sum\limits_{j\in\textbf{k}_{i_0}}\big(V(x_j)-2\int_{\sigma\backslash\sigma_{i_0}}\ln|x_j-y|\rho(y)dy\big)\Big]\prod\limits_{u<v\atop
	 			u,v\in\textbf{k}_{i_0}}|x_u-x_v|^\beta
	 		\mathds{1}_{\tilde\Omega_N(i_0)}(\bar x)\prod\limits_{j\in\textbf{k}_{i_0}}dx_j}{Q_{(k_1,\ldots,k_q)}^{(i_0)}}\\
	 	=&\mathbb P^{\nu^{(i_0,k_{i_0},NR_{i_0}/k_{i_0})}}(\tilde\Omega_N(i_0)).
	 \end{align*} 
	 which implies
	 \begin{align*}
	 	\P^{\mu_r}(A_N'(i_0)\cap\Omega_N(\epsilon_*))\le
	 	\dfrac{\sum\limits_{k_1+\cdots+k_q=N\atop|k_i-NR_i|\le
	 			N^{\epsilon_*},\forall
	 			i}Q_{(k_1,\ldots,k_q)}\cdot\mathbb P^{\nu^{(i_0,k_{i_0},NR_{i_0}/k_{i_0})}}(\tilde\Omega_N(i_0))}{\sum\limits_{k_1+\cdots+k_q=N}Q_{(k_1,\ldots,k_q)}}\quad\text{(see \eqref{eq:271})}.
	 \end{align*}
	 This together with \eqref{eq:272} and the third conclusion of Lemma \ref{lemma:lemma_on_k_i} tell us 
	 $$\P^{\mu_r}(A_N'(i_0)\cap\Omega_N(\epsilon_*))\le\exp(-k_{i_0}^{c_0})\le\exp(-N^{c_0/2})\quad\text{if }N>N_0$$
	 where $c_0>0$ is the same as in \eqref{eq:272}. Notice that $\kappa$ depends only on $V$. Using Lemma \ref{lemma:measures} we complete the proof of the first conclusion of Theorem \ref{thm:main_thm}.
	 
	 \subsection{Rigidity near the extreme edges}
	 We prove the rigidity near the extreme edges in the similar way as for the bulk. Set 
	 $$A_N''=\{x\in\Sigma_\kappa^{(N)}|\exists
	 k\in[1,\alpha N]\text{
	 	such that }|x_k-\eta_k|>N^{-\frac{2}{3}+\epsilon}\cdot k^{-1/3}\};$$
	 $$A_N''':=\{x\in\Sigma_\kappa^{(N)}|\exists
	 k\in[(1-\alpha)N, N]\text{
	 	such that }|x_k-\eta_k|>N^{-\frac{2}{3}+\epsilon}(N+1-k)^{-1/3}\}.$$
	 Similarly, to prove the second conclusion of Theorem \ref{thm:main_thm}, we only need to show that $\P^{\mu_r}(A_N''\cap\Omega_N(\epsilon_*))$ and $\P^{\mu_r}(A_N'''\cap\Omega_N(\epsilon_*))$ are exponentially small.
	 
	 Similarly as above, if
	 \begin{align}\label{eq:254}
	 	\Upsilon(k_1,\ldots,k_q)\mathds{1}_{A_N''\cap\Omega_N(\epsilon_*)}(x)=1
	 \end{align}
	 then obviously for $N>N_0$ we have $k_i\in(0.99NR_i,1.01NR_i)$  $(\forall1\le i\le q)$,
	 thus there exists $k\in[1,\alpha N]\subset [1,\frac{\alpha}{0.99R_1}k_1]$ with
	 $$|x_k-\theta(1,k_1,k)|\ge|x_k-\eta_k|-|\eta_k-\theta(1,k_1,k)|\ge N^{-\frac{2}{3}+\epsilon} k^{-1/3}-C_2N^{-\frac{5}{3}+\epsilon_*}k^{2/3}> k_1^{-\frac{2}{3}+\frac{1}{2}\epsilon} k^{-1/3}.$$
	 Here $C_2>0$ is a constant and we used Lemma \ref{lemma:lemma_on_k_i}.
	 Therefore  if $N>N_0$ then
	 \begin{multline*}
	 	\mathbb P^{\mu_r}(A_N''\cap\Omega_N(\epsilon_*))\le\dfrac{\sum\limits_{k_1+\cdots+k_q=N\atop |k_i-NR_i|\le N^{\epsilon_*},\forall i}Q_{(k_1,\ldots,k_q)}\cdot\Phi^{(A_N'')}_{k_1,\ldots,k_q}}{\sum\limits_{k_1+\cdots+k_q=N}Q_{(k_1,\ldots,k_q)}}\\
	 	\le\dfrac{\sum\limits_{k_1+\cdots+k_q=N\atop |k_i-NR_i|\le N^{\epsilon_*},\forall i}Q_{(k_1,\ldots,k_q)}\cdot\mathbb P^{\nu^{(1,k_1,NR_1/k_1)}}\Big(\exists k\in[1,\frac{\alpha k_1}{0.99R_1}]\text{ s.t. }|x_k-\theta(1,k_1,k)|>k_1^{-\frac{2}{3}+\frac{1}{2}\epsilon}k^{-1/3}\Big)}{\sum\limits_{k_1+\cdots+k_q=N}Q_{(k_1,\ldots,k_q)}}
	 \end{multline*}
	 Notice that $\epsilon_*$ is a constant.	Using \eqref{eq:278}, we have that if $N>N_0$ 
	 \begin{align}\label{eq:279}
	 	\mathbb P^{\mu_r}(A_N''\cap\Omega_N(\epsilon_*))\le\exp(-k_1^{c_0})\le\exp(-N^{c_0/2})
	 \end{align}
	 where $c_0>0$ is the same as in \eqref{eq:278}. Similarly we can prove that \eqref{eq:279} is true if we replace $A_N''$ by $A_N'''$.
	 
	 So for $N>N_0$,
	 \begin{multline*}
	 	\P^{\mu_r}\Big(\Omega_N(\epsilon_*)\cap\Big\{x\in\Sigma_\kappa^{(N)}|\exists
	 	k\in[1,\alpha N]\cup[(1-\alpha)N, N]\text{
	 		s.t. }|x_k-\eta_k|>N^{-\frac{2}{3}+\epsilon}\hat k^{-1/3}\Big\}\Big)\\
	 	\le2\exp(-N^{c_0/2}).
	 \end{multline*}
	 Notice that $\kappa$ depends only on $V$. Using Lemma \ref{lemma:measures} we complete the proof of the second conclusion of Theorem \ref{thm:main_thm}.
	 
	 \section{Concentration for linear eigenvalue statistics: the proof of Theorem \ref{thm:fluctuation_multi_cut}}\label{appendix:proof_of_fluctuation_for_initial_model}
	 
	 To prove Theorem \ref{thm:fluctuation_multi_cut} we need to use the following theorem given by Theorem 2 of \cite{Shcherbina_fluctuation}. Similar result is obtained but formulated in a slightly different way in Theorem 1.6 of \cite{BG_multi_cut} . Recall that the support of the equilibrium measure $\rho(t)dt$ is $\sigma=[A_1,B_1]\cup\ldots\cup[A_q,B_q]$.

	 \begin{theorem}[Theorem 2 of \cite{Shcherbina_fluctuation}]\label{thm:fluctuation}
	 	Suppose $0<\epsilon<\frac{1}{10}\min\limits_{2\le i\le q}(A_i-B_{i-1})$. Suppose  $h_N(x)\in C^\infty(\mathbb R)$ and:
	 	\begin{itemize}
	 		\item $\text{supp}(h_N)\subset \cup_{i=1}^q[A_i-\epsilon,B_i+\epsilon]$;
	 		\item $\sup\limits_{N\in\mathbb N}\|h_N^{(i)}\|_\infty<\infty$ for $i\in\{0,1,\ldots,6\}$.
	 	\end{itemize}
	 	Then we have
	 	\begin{multline}\label{eq:2222}
	 		\E^\mu\Big[\exp\Big(\sum h_N(x_i)-N\int h_N(t)\rho(t)dt\Big)\Big]=e^{(\beta-2) L_1(h_N)+\beta L_2(h_N,h_N)}(1+\mathcal E_N)\\
	 		\times\dfrac{\sum\limits_{(v_1,\ldots,v_q)\in W_N}\exp\Big(-\frac{\beta}{2}(\mathcal Q^{-1}\Delta\bar v,\Delta\bar v)+\frac{\beta}{2}(\Delta\bar v,I[h_N])+(\frac{\beta}{2}-1)(\Delta\bar v,c_V)\Big)}{\sum\limits_{(v_1,\ldots,v_q)\in W_N}\exp\Big(-\frac{\beta}{2}(\mathcal Q^{-1}\Delta\bar v,\Delta\bar v)+(\frac{\beta}{2}-1)(\Delta\bar v,c_V)\Big)}
	 	\end{multline}
	 	where 
	 	\begin{itemize}
	 		\item  $W_N=\{(v_1,\ldots,v_q)\in \mathbb Z^q|\sum_{i=1}^q v_i=\sum_{i=1}^q\{NR_i\}\}$ where $\{\cdot\}$ denotes the fractional part: $\{x\}=x-\lfloor x \rfloor$ and $\lfloor x \rfloor$ is the largest integer no more than $x$;
	 		\item $(\cdot,\cdot)$ denotes the inner product in $\mathbb R^q$;
	 		\item $|\mathcal E_N|\le C_1N^{-w_0}$ where  $C_1>0$  and  $w_0>0$ are constants;
	 		\item $\mathcal Q$ is a $q\times q$ invertible  positive definite matrix determined by $V$;
	 		\item $I[h_N]$ is a $q$-dimensional vector  and each component is bounded by $C_2\|h_N\|_\infty$ where $C_2>0$ is a constant;
	 		\item the $q$-dimensional vector $c_V$ is determined by $V$;
	 		\item $\Delta\bar v=(v_1-\{NR_1\},\ldots v_q-\{NR_q\})\in\mathbb R^q$;
	 		\item $L_1(\cdot)$ is a linear functional. The map $L_1$ is independent of $N$ and $\beta$.
	 		\item $L_2(\cdot,\cdot)$ is a bilinear functional. The map $L_2$ is independent of $N$ and $\beta$. 
	 	\end{itemize}
	 \end{theorem}
	 \begin{remark}
	 	\cite{Shcherbina_fluctuation} provides explicit constructions of $L_1$ and $L_2$. See Page 1008-1009 of \cite{Shcherbina_fluctuation}. But we only use the two facts:  (i) they are linear and bilinear respectively; (ii) they are independent of $N$ and $\beta$.
	 \end{remark}
	 
	 \begin{remark}\label{remark:consistent}
	 	The term 
	 	\begin{align}\label{eq:2228}
	 		\sum\limits_{(v_1,\ldots,v_q)\in W_N}\exp\Big(-\frac{\beta}{2}(\mathcal Q^{-1}\Delta\bar v,\Delta\bar v)+\frac{\beta}{2}(\Delta\bar v,I[h_N])+(\frac{\beta}{2}-1)(\Delta\bar v,c_V)\Big)
	 	\end{align}
	 	on the right hand side of \eqref{eq:2222} is from Theorem 2 of \cite{Shcherbina_fluctuation}, but the corresponding term  in Theorem 1.6 of \cite{BG_multi_cut} is defined in a different way. As summarized on Page 19 of \cite{Lambert}, the term in \cite{BG_multi_cut} corresponding to \eqref{eq:2228} is
	 	\begin{align}\label{eq:2229}
	 		\sum\limits_{\bar k\in W_N'}\exp\Big(-\frac{1}{2}(\tau\bar k,\bar k)+(C_{\beta,V}+\frac{\beta}{2}\tilde I[h_N],\bar k)\Big)
	 	\end{align}
	 	where $\tau$ is a $(q-1)\times(q-1)$ positive definite matrix, $C_{\beta,V}$ is a  $(q-1)$-dimensional constant vector, $\tilde I$ is a $\mathbb R^{q-1}$-valued linear functional and
	 	$$W_N'=\{(k_1,\ldots,k_{q-1})|k_i-\{NR_i\}\in\mathbb Z \text{ for all }i\in[1,q-1]\}.$$
	 	We notice that \eqref{eq:2228} is consistent with \eqref{eq:2229} through
	 	\begin{align}\label{eq:2230}
	 		\tau_{ij}=\beta(\mathcal Q^{-1}_{ij}+\mathcal Q^{-1}_{qq}-\mathcal Q^{-1}_{qj}-\mathcal Q^{-1}_{qi})\quad (i,j\in\{1,\ldots,q-1\})
	 	\end{align}
	 	$k_i=\{NR_i\}-v_i$, 
	 	$(C_{\beta,V})_i=(\frac{\beta}{2}-1)[(c_V)_q-(c_V)_i]$ and $(\tilde I[h_N])_i=(I[h_N])_q-(I[h_N])_i$. The matrix given by \eqref{eq:2230} is positive definite because: (i) $\mathcal Q^{-1}$ is positive definite because $\mathcal Q$ is positive definite; (ii) for any $(x_1,\ldots,x_{q-1})\in\mathbb R^{q-1}$: 
	 	$$(x_1,\ldots,x_{q-1})\tau(x_1,\ldots,x_{q-1})^T=\beta\cdot(x_1,\ldots,x_{q-1},-\sum_{i=1}^{q-1}x_i) \mathcal Q^{-1}(x_1,\ldots,x_{q-1},-\sum_{i=1}^{q-1}x_i)^T.$$
	 \end{remark}
	 
	 \begin{lemma}\label{lemma:bound}
	 	Suppose the assumptions of Theorem \ref{thm:fluctuation} are true. Then the following results hold.
	 	\begin{enumerate}
	 		\item There exists a constant $C_3>0$ such that
	 		\begin{align}\label{eq:2242}
	 			\dfrac{\sum\limits_{(v_1,\ldots,v_q)\in W_N}\exp\Big(-\frac{\beta}{2}(\mathcal Q^{-1}\Delta\bar v,\Delta\bar v)+\frac{\beta}{2}(\Delta\bar v,I[h_N])+(\frac{\beta}{2}-1)(\Delta\bar v,c_V)\Big)}{\sum\limits_{(v_1,\ldots,v_q)\in W_N}\exp\Big(-\frac{\beta}{2}(\mathcal Q^{-1}\Delta\bar v,\Delta\bar v)+(\frac{\beta}{2}-1)(\Delta\bar v,c_V)\Big)}\le e^{C_3(1+\|h_N\|_\infty^2)}.
	 		\end{align}
	 		\item There exists a constant $C_4>0$ such that
	 		\begin{align}\label{eq:2232}
	 			|L_2(h_N,h_N)|\le C_4\cdot \|h_N'\|_\infty^2.
	 		\end{align}
	 		\item Suppose there exists an constant $\epsilon'\in(0,\epsilon)$ such that each $h_N$ can be analytically extended to $O:=\{z\in\mathbb C|\text{dist}(z,\sigma)<\epsilon'\}$. Then \begin{align}\label{eq:2231}
	 			|(\beta-2)L_1(h_N)|\le C_5\cdot\sup\limits_{z\in O}|h_N(z)|
	 		\end{align}
	 		where $C_5>0$ is a constant.
	 	\end{enumerate}
	 \end{lemma}
	 
	 \begin{proof}
	 	\begin{enumerate}
	 		\item  To prove \eqref{eq:2242}, we notice that $\mathcal L^{-1}$ is positive definite since $\mathcal L$ is positive definite. So $-\frac{\beta}{2}(\mathcal Q^{-1}x,x)\le -C_1'\|x\|_2^2$ (for all $x\in\mathbb R^q$) for some constant $C_1'>0$. Since the components of $I[h_N]$ is bounded by $\|h_N\|_\infty$ multiplied by a constant (see Theorem \ref{thm:fluctuation}), we have that
	 		$$-\frac{\beta}{2}(\mathcal Q^{-1}\Delta\bar v,\Delta\bar v)+\frac{\beta}{2}(\Delta\bar v,I[h_N])+(\frac{\beta}{2}-1)(\Delta\bar v,c_V)\le-C_1'\|\Delta\bar v\|_2^2+C_2'\|\Delta\bar v\|_2(\|h_N\|_\infty+1)$$
	 		where $C_2'>0$ is a constant. This together with the fact that $W_N\subset\mathbb Z^q$ imply that the numerator on the left hand side of \eqref{eq:2242} is bounded above by 
	 		\begin{multline}\label{eq:2243}
	 			\sum\limits_{(v_1,\ldots,v_q)\in \mathbb Z^q}\exp\Big(-C_1'\|\Delta\bar v\|_2^2+C_2'\|\Delta\bar v\|_2(\|h_N\|_\infty+1)\Big)\\
	 			\le \prod_{i=1}^q\sum_{v_i\in\mathbb Z}\exp\Big(-C_1'(\Delta\bar v)_i^2+C_2'(\|h_N\|_\infty+1)|(\Delta\bar v)_i|\Big).
	 		\end{multline}
	 		Since $(\Delta\bar v)_i=v_i-\{NR_i\}$, we see that $|(\Delta\bar v)_i|\le|v_i|+1$ and that if $|v_i|\ge2$ then $|(\Delta\bar v)_i|\ge |v_i|-1\ge|v_i|/2$. So there exists a constant $C_3'>0$  such that if $|v_i|\ge C_3'(\|h_N\|_\infty+1)$, then
	 		\begin{multline}
	 			-C_1'(\Delta\bar v)_i^2+C_2'(\|h_N\|_\infty+1)|(\Delta\bar v)_i|
	 			\le-\frac{C_1'}{4}v_i^2+C_2'(\|h_N\|_\infty+1)(1+|v_i|)\\
	 			\le-\frac{C_1'}{4}v_i^2+2C_2'(\|h_N\|_\infty+1)|v_i| 
	 			\le -C_1'v_i^2/5.
	 		\end{multline}
	 		Therefore the right hand side of \eqref{eq:2243} is bounded above by
	 		\begin{multline} \label{eq:2244}
	 			\prod_{i=1}^q\Big(\sum_{v_i\in\mathbb Z \text{ and}\atop|v_i|< C_3'(\|h_N\|_\infty+1)}e^{C_1'(\|h_N\|_\infty+1)|(\Delta\bar v)_i|}+\sum_{v_i\in\mathbb Z\text{ and}\atop|v_i|\ge C_3'(\|h_N\|_\infty+1)}e^{-C_1'v_i^2/5}\Big)\\
	 			\le \prod_{i=1}^q\Big(\sum_{v_i\in\mathbb Z \text{ and}\atop|v_i|< C_3'(\|h_N\|_\infty+1)}e^{C_1'(\|h_N\|_\infty+1)(1+C_3'(\|h_N\|_\infty+1))}+\sum_{v_i\in\mathbb Z}e^{-C_1'v_i^2/5}\Big)\quad(\text{since }|(\Delta\bar v)_i|\le|v_i|+1)\\
	 			\le e^{C_4'(1+\|h_N\|_\infty^2)}
	 		\end{multline}
	 		for some constant $C_4'>0$. Here we used the fact that $\sum_{v_i\in\mathbb Z}e^{-C_1'v_i^2/5}$ is a constant in the last inequality.  Since $\sum_{i=1}^qNR_i=N$, we see that $\sum_{i=1}^q\{NR_i\}=0$, so $(0,\ldots,0)\in W_N$ and the denominator  on the left  hand side of \eqref{eq:2242} is bounded   below by 
	 		\begin{multline}\label{eq:2245}
	 			\exp\Big(-\frac{\beta}{2}\sum_{i,j=1}^q\{NR_i\}\{NR_j\}\mathcal Q^{-1}_{ij}+\sum_{i=1}^q(1-\frac{\beta}{2})\{NR_i\}(c_V)_i\Big)\\
	 			\ge \exp\Big(-\frac{\beta}{2}\cdot q^2\cdot \max_{i,j}|\mathcal Q^{-1}_{ij}|-q(1+\frac{\beta}{2})\max_i|(c_V)_i|\Big).
	 		\end{multline}
	 		The right ahnd side of \eqref{eq:2245} is a constant. By \eqref{eq:2243}, \eqref{eq:2244} and \eqref{eq:2245} we complete the proof of the \eqref{eq:2242}.

	 		\item Now we prove \eqref{eq:2232}. Notice that   $L_2$ is an $N$-independent map, so in order to prove \eqref{eq:2232} we only need to prove it for an $N$-independent test function. Suppose  $f(x)$ is an $N$-independent smooth function such that $\text{supp}(f)\subset \cup_{i=1}^q[A_i-\epsilon,B_i+\epsilon]$. By \eqref{eq:2222} and \eqref{eq:2242},
	 		\begin{align}\label{eq:2236}
	 			\E^\mu\Big[\exp\Big(\sum f(x_i)-N\int f(t)\rho(t)dt\Big)\Big]\le c_1(f)
	 		\end{align}
	 		where $c_1(f)>0$ is an $N$-independent but $f$-dependent quantity.
	 		
	 		By Bolzano-Weierstrass' Theorem we can find a subsequence $N_1,N_2,\ldots$   of $1,2,\ldots$ such that $u_i:=\lim_{k\to\infty}\{N_kR_i\}$ exists  for each $i\in\{1,\ldots,q\}$. 
	 		Then as $k\to\infty$, along this subsequence we have
	 		\begin{multline}
	 			\lim\limits_{k\to\infty}	\E^\mu\Big[\exp\Big(\sum_{i=1}^{N_k} f(x_i)-N_k\int f(t)\rho(t)dt\Big)\Big]=e^{(\beta-2) L_1(f)+\beta L_2(f,f)}\\
	 			\times\dfrac{\sum\limits_{(v_1,\ldots,v_q)}\exp\Big(-\frac{\beta}{2}(\mathcal Q^{-1}(\bar v-\bar u),(\bar v-\bar u))+\frac{\beta}{2}((\bar v-\bar u),I[f])+(\frac{\beta}{2}-1)((\bar v-\bar u),c_V)\Big)}{\sum\limits_{(v_1,\ldots,v_q)}\exp\Big(-\frac{\beta}{2}(\mathcal Q^{-1}(\bar v-\bar u),(\bar v-\bar u))+(\frac{\beta}{2}-1)((\bar v-\bar u),c_V)\Big)}
	 		\end{multline}
	 		where $\mu=\mu(N_k)$, $(\bar v-\bar u)=(v_1-u_1,\ldots,v_q-u_q)$ and $(v_1,\ldots,v_q)$ runs over $$K_u:=\{(v_1,\ldots,v_q)\in \mathbb Z^q|\sum_{i=1}^q v_i=\sum_{i=1}^qu_i\}.$$ Replacing $f(x)$ by $tf(x)$ we see that the Laplace transform of $\sum_{i=1}^{N_k} f(x_i)-N_k\int f(t)\rho(t)dt$ converges to the Laplace transform of $X_f+Y_f$ where $X_f$ and $Y_f$ are $independent$ random variables such that
	 		\begin{align}\label{eq:2246}
	 			X_f\sim\mathcal N\Big((2-\beta)L_1(f),2\beta L_2(f,f)\Big)
	 		\end{align}
	 		\begin{align}\label{eq:2247}
	 			Y_f=(\bar w-\bar u,\frac{\beta}{2}I[f])
	 		\end{align}
	 		and the distribution of the random vector $\bar w$ is:
	 		$$\mathbb P(\bar w=\bar v)=\dfrac{\exp\Big(-\frac{\beta}{2}(\mathcal Q^{-1}(\bar v-\bar u),(\bar v-\bar u))+(\frac{\beta}{2}-1)((\bar v-\bar u),c_V)\Big)}{\sum\limits_{\bar v\in K_u}\exp\Big(-\frac{\beta}{2}(\mathcal Q^{-1}(\bar v-\bar u),(\bar v-\bar u))+(\frac{\beta}{2}-1)((\bar v-\bar u),c_V)\Big)}\quad\forall \bar v\in K_u.$$
	 		So  $\sum_{i=1}^{N_k} f(x_i)-N_k\int f(t)\rho(t)dt$ converges in distribution to $X_f+Y_f$ as $k\to\infty$. 
	 		
	 		We remark that the convergence along the subsequence was first observed by Pastur \cite{Pastur}. The method of studying the convergence along the subsequence was also used in \cite{Bekerman} and \cite{BG_multi_cut}. The distributions of $X_f$ and $Y_f$ are described in \cite{BG_multi_cut}. A good survey of this topic can be found in Section 4.1 of \cite{Lambert}.
	 		\begin{lemma}\label{lemma:var}
	 			$$\lim_{k\to\infty} \text{Var}\Big(\sum_{i=1}^{N_k} f(x_i)-N_k\int f(t)\rho(t)dt\Big)=\text{Var}(X_f+Y_f)$$
	 		\end{lemma}
	 		\begin{proof}
	 			Let $a_k=\sum_{i=1}^{N_k} f(x_i)-N_k\int f(t)\rho(t)dt$. In the following part of the proof of Lemma \ref{lemma:var}, when we consider $\E^\mu[\text{a function of }a_k]$ or $\P^\mu[\text{an event about }a_k]$ we mean $\mu=\mu(N_k)$. For any $y>0$, by \eqref{eq:2236},
	 			\begin{align}\label{eq:2237}
	 				\P^\mu(a_k>y)\le e^{-y}\E^\mu[e^{a_k}]\le c_1(f)e^{-y}.
	 			\end{align}
	 			Using \eqref{eq:2236} with $f$ replaced by $-f$ we have
	 			\begin{align}\label{eq:2238}
	 				\P^\mu(a_k<-y)=\P^\mu(-a_k>y)\le e^{-y}\E^\mu[e^{-a_k}]\le c_1(f)e^{-y}.
	 			\end{align}
	 			By \eqref{eq:2237} and \eqref{eq:2238},
	 			\begin{align}\label{eq:2239}
	 				\E^{\mu}[a_k^4]=4\int_0^\infty y^3\P^\mu(|a_k|>y)dy\le 8c_1(f)\int_0^\infty y^3e^{-y}dy=48c_1(f).
	 			\end{align}
	 			In \eqref{eq:2239} we used the fact that if $W\ge0$ and $p>0$ then $\E[W^p]=p\int_0^\infty y^{p-1}\P(W>y)dy$. See, for example, Lemma 2.2.13 of \cite{Durrett}.
	 			Since $a_k\to X_f+Y_f$ in distribution as $k\to\infty$, according to the Skorokhod's Theorem, there exist random variables $b_1$, $b_2$,\ldots and $b$ all defined on a same probability space such that
	 			\begin{itemize}
	 				\item $b_k\to b$ almost surely	as $k\to\infty$;
	 				\item $b_k$ has the same distribution as $a_k$ and $b$ has the same distribution as $X_f+Y_f$.
	 			\end{itemize}
	 			By \eqref{eq:2239}, we have $\E[b_k^4]\le48c_1(f)$. So for an arbitrarily large $M>0$ we have:
	 			\begin{multline*}
	 				\E[|b_k|\mathds1_{|b_k|>M}]\le \big(\E[b_k^4]\big)^{\frac{1}{4}}\big(\P(|b_k|>M)\big)^{\frac{3}{4}}\\
	 				\le(48c_1(f))^{\frac{1}{4}}\big(\P(|b_k|>M)\big)^{\frac{3}{4}}\rightarrow (48c_1(f))^{\frac{1}{4}}\big(\P(|b|>M)\big)^{\frac{3}{4}}\quad\text{as }k\to\infty,
	 			\end{multline*}
	 			and
	 			\begin{multline*}
	 				\E[b_k^2\mathds1_{b_k^2>M}]\le \sqrt{\E[b_k^4]\P(b_k^2>M)}\\
	 				\le\sqrt{48c_1(f)}\sqrt{\P(b_k^2>M)}\rightarrow \sqrt{48c_1(f)}\sqrt{\P(b^2>M)}\quad\text{as }k\to\infty
	 			\end{multline*}
	 			and they imply that both $\{b_k\}$ and $\{b_k^2\}$ are uniformly  integrable. This together with the fact that $b_k\to b$ almost surely yield: 
	 			\begin{align}\label{eq:2240}
	 				\lim_{k\to\infty}\E[b_k]=\E[b]\quad\text{and}\quad \lim_{k\to\infty}\E[b_k^2]=\E[b^2].
	 			\end{align}
	 			Since $b_k$ has the same distribution as $a_k$ and $b$ has the same distribution as $X_f+Y_f$, \eqref{eq:2240} must be true with $b_k$ replaced by $a_k$ and $b$ replaced by $X_f+Y_f$. So the proof of  Lemma \ref{lemma:var} is complete.
	 		\end{proof}
	 		To prove \eqref{eq:2232}, we use the fact that $L_2(f,f)$ is independent of $\beta$ (see Theorem \ref{thm:fluctuation}). Let $\beta=2$, then
	 		$$X_f\sim\mathcal N(0,4 L_2(f,f))$$
	 		It is well know that when $\beta=2$ we have
	 		\begin{align}\label{eq:2241}
	 			\text{Var}\Big(\sum f(x_i)-N\int f(t)\rho(t)dt\Big)\le C\cdot \sup_{x\in\mathbb R}|f'(x)|^2
	 		\end{align}
	 		where $C>0$ is a constant. This result can be found in Page 7-8 of \cite{Pastur}. Since $X_f$ and $Y_f$ are independent, we have by Lemma \ref{lemma:var} and \eqref{eq:2241}:
	 		\begin{multline*}
	 			4L_2(f,f)=\text{Var}(X_f)\le\text{Var}(X_f+Y_f)\\
	 			=\lim_{k\to\infty }\text{Var}\Big(\sum f(x_i)-N\int f(t)\rho(t)dt\Big)\le C\cdot \sup_{x\in\mathbb R}|f'(x)|^2.
	 		\end{multline*}
	 		So \eqref{eq:2232} is proved.
	 		
	 		\item Now we prove \eqref{eq:2231}. Similarly as above,  we only need to prove \eqref{eq:2231} for an $N$-independent test function since $L_1$ is an $N$-independent functional. Suppose  $g(x)$ is an $N$-independent smooth function such that
	 		\begin{itemize}
	 			\item $\text{supp}(g)\subset \cup_{i=1}^q[A_i-\epsilon,B_i+\epsilon]$;
	 			\item  $g(x)$ can be analytically extended to $O=\{z\in\mathbb C|\text{dist}(z,\sigma)<\epsilon'\}$.
	 		\end{itemize}
	 		Similarly as above, choose the subsequence $N_1,N_2,\ldots$ such that $u_i:=\lim_{k\to\infty}\{NR_i\}$ exists for $1\le i\le q$. Then as $k\to\infty$, $\sum_{i=1}^{N_k} g(x_i)-N_k\int g(t)\rho(t)dt$ converges in distribution to $X_g+Y_g$ where $X_g$ and $Y_g$ are independent random variables and their distributions are defined by  \eqref{eq:2246} and \eqref{eq:2247} with $f$ replaced by $g$.
	 		
	 		According to Proposition 5.2 and (8.19) of \cite{BG_multi_cut} with $s=-{\rm i}t$, there exists a function $W_1^{\{0\}}(\xi)$ such that
	 		\begin{itemize}
	 			\item $W_1^{\{0\}}(\xi)$ is determined by $V$ and is analytic on $\mathbb C\backslash\sigma$;
	 			\item the mean value of $X_g$ is 
	 			\begin{align}\label{eq:2233}
	 				\oint \frac{d\xi}{2\pi{\rm i}}g(\xi)W_1^{\{0\}}(\xi)
	 			\end{align}
	 			where the integral is along the contours $\cup_{i=1}^q\mathcal C_i$ and each $\mathcal C_i$ is contained by $O$ and  encloses $[A_i,B_i]$
	 		\end{itemize}
	 		By \eqref{eq:2233},
	 		$|(\beta-2) L_1(g)|$, i.e., the absolute value of the mean value of $X_g$, is bounded by $\sup\limits_{z\in O}|g(z)|$ multiplied by a constant. This yields \eqref{eq:2231}.
	 	\end{enumerate}
	 \end{proof}
	 
	 \begin{remark}
	 	We used results from \cite{BG_multi_cut} to prove  \eqref{eq:2231}. The Hypothesis 1.2 and Hypothesis 1.3 in \cite{BG_multi_cut} are automatically satisfied for our model because  the potential $V$ is analytic and is independent of $N$.
	 \end{remark}

	 \begin{proof}[Proof of Theorem \ref{thm:fluctuation_multi_cut}]
	 	
	 	Suppose $0<\epsilon<\frac{1}{10}\min\limits_{2\le i\le q}(A_i-B_{i-1})$. Suppose $\varphi(x)$ is a smooth function such that: (i) $\varphi(x)=1$ if $\text{dist}(x,\sigma)\le\epsilon/2$; (ii) $\varphi(x)=0$ if $\text{dist}(x,\sigma)\ge\epsilon$;  (iii) $\varphi(x)\in[0,1]$ for all $x\in\mathbb R$.
	 	
	 	Let $g_N(x)=h_N(x)\varphi(x)$. So by \eqref{eq:2207}, there are constants $N_0>0$  and $C_1'>0$  such that if $N>N_0$ then
	 	\begin{multline}\label{eq:2251}
	 		\P^\mu\Big(\Big|\sum\limits_{i=1}^Nh_N(x_i)-N\int h_N(t)\rho(t)dt\Big|>N^{w_1}\Big)\\
	 		\le 	\P^\mu\Big(\Big|\sum\limits_{i=1}^Ng_N(x_i)-N\int g_N(t)\rho(t)dt\Big|>N^{w_1}\Big)+	\P^\mu\Big(\exists i\text{ such that }\text{dist}(x_i,\sigma)>\epsilon/2\Big)\\
	 		\le\P^\mu\Big(\Big|\sum\limits_{i=1}^Ng_N(x_i)-N\int g_N(t)\rho(t)dt\Big|>N^{w_1}\Big)+e^{-C_1'N}. 
	 	\end{multline}
	 	Choose $\epsilon'\in(0,\epsilon/2)$ small enough such that 
	 	\begin{align}\label{eq:2248}
	 		O':=\{z\in\mathbb C|\text{dist}(z,\sigma)<\epsilon'\}\subset O.
	 	\end{align}
	 	By our assumption $h_N$ can be analytically extended to $O'$.
	 	According the construction of $g_N$, it can  be analytically extended in the same way as $h_N$ to $O'$. Moreover, 
	 	$$|g_N(z)|=|h_N(z)|\le C_b\quad\quad \text{for all }z\in O'\text{ and }N\ge1$$
	 	$$|g_N^{(i)}(x)|\le C_2'\max_{0\le k\le6}\|h_N^{(k)}\|_\infty\le C_2'\cdot C_b\quad\quad\text{for all } x\in\mathbb R, 0\le i\le6\text{ and }N\ge1$$
	 	where $C_2'$ is a constant.
	 	
	 	Using Theorem \ref{thm:fluctuation} and Lemma \ref{lemma:bound} for $g_N$:
	 	\begin{multline}\label{eq:2249}
	 		\P^\mu\Big(\sum\limits_{i=1}^Ng_N(x_i)-N\int g_N(t)\rho(t)dt>N^{w_1}\Big)\le e^{-N^{w_1}}\E^\mu\Big[e^{\sum\limits_{i=1}^Ng_N(x_i)-N\int g_N(t)\rho(t)dt}\Big]\\
	 		\le e^{-N^{w_1}}\exp\Big(C_5\cdot\sup_{z\in O'}|g_N(z)|+\beta\cdot C_4\|g_N'\|_\infty^2\Big)(1+C_1)\exp\Big(C_3(1+\|g_N\|_\infty^2)\Big)\\
	 		\le e^{-N^{w_1}}e^{C_5C_b+\beta C_4C_2'^2C_b^2}(1+C_1)e^{C_3(1+C_2'^2C_b^2)}\le e^{-\frac{1}{2}N^{w_1}}
	 	\end{multline}
	 	when $N>N_0$. The constants $C_1$, $C_3$, $C_4$ and $C_5$ are defined in Theorem \ref{thm:fluctuation} and Lemma \ref{lemma:bound}. Using Theorem \ref{thm:fluctuation} and Lemma \ref{lemma:bound} in the same way for $-g_N$ we have:
	 	\begin{multline} \label{eq:2250}
	 		\P^\mu\Big(\sum\limits_{i=1}^Ng_N(x_i)-N\int g_N(t)\rho(t)dt<-N^{w_1}\Big)\\
	 		=\P^\mu\Big(-\sum\limits_{i=1}^Ng_N(x_i)-N\int (-g_N(t))\rho(t)dt>N^{w_1}\Big)\le e^{-\frac{1}{2}N^{w_1}}
	 	\end{multline}
	 	for $N>N_0$. Using \eqref{eq:2251}, \eqref{eq:2249} and \eqref{eq:2250} we complete the proof.
	 \end{proof}
	 
	 \section*{Acknowledgement}
	 I thank my PhD advisor Mark Adler because this paper is based on my thesis. I also thank the following mathematicians for helpful communication: Paul Bourgade, L\'aszl\'o Erd\H{o}s, Alice Guionnet, Kurt Johansson, Kevin Schnelli, Mariya Shcherbina, Xin Sun, Horng-Tzer Yau and Zhiyuan Zhang.

 \end{document}